\documentclass[11pt]{article}

\usepackage{amssymb}
\usepackage[nosumlimits,nonamelimits]{amsmath}
\usepackage{amsxtra}
\usepackage{enumitem}
\usepackage[text={0.75\paperwidth,0.75\paperheight}]{geometry} %for page layout
\usepackage{hyperref}
\hypersetup{colorlinks=true,allcolors=red}
\usepackage[bbgreekl]{mathbbol} %blackboard bold for greek letters
\usepackage{mathtools}
\usepackage{multicol}
\usepackage[amsmath,amsthm,thmmarks,hyperref]{ntheorem} %replaces amsthm with better handling of \qed placement
\usepackage{titlesec}
\usepackage[all]{xy}

\allowdisplaybreaks

\usepackage[parfill]{parskip}
\begingroup
\makeatletter
   \@for\theoremstyle:=definition,remark,plain\do{%
     \expandafter\g@addto@macro\csname th@\theoremstyle\endcsname{%
        \addtolength\thm@preskip\parskip
     }%
   }
\endgroup

\DeclareMathAlphabet{\mathpzc}{OT1}{pzc}{m}{it} %allows \mathpzc command for Zapf Chancery alphabet

\DeclareFontFamily{U}{MnSymbolC}{}
\DeclareSymbolFont{mnsymbols}{U}{MnSymbolC}{m}{n}
\DeclareFontShape{U}{MnSymbolC}{m}{n}{
    <-6>  MnSymbolC5
   <6-7>  MnSymbolC6
   <7-8>  MnSymbolC7
   <8-9>  MnSymbolC8
   <9-10> MnSymbolC9
  <10-12> MnSymbolC10
  <12->   MnSymbolC12}{}
 \DeclareMathSymbol{\boxtimes}{2}{mnsymbols}{117}
\DeclareSymbolFont{symbolsC}{U}{txsyc}{m}{n}
\DeclareMathSymbol{\multimapinv}{\mathrel}{symbolsC}{18}
\DeclareMathSymbol{\multimapdot}{\mathrel}{symbolsC}{20}
\DeclareMathSymbol{\multimapdotinv}{\mathrel}{symbolsC}{21}

\xyoption{tips}
\SelectTips{cm}{10}
\UseTips

\DeclareSymbolFont{symbolsA}{U}{txsya}{m}{n}
\DeclareSymbolFont{symbolsC}{U}{txsyc}{m}{n}
\DeclareMathSymbol{\multimap}{\mathrel}{symbolsA}{40}
\DeclareMathSymbol{\multimapinv}{\mathrel}{symbolsC}{18}

\theoremstyle{change}
\theoremseparator{.}
\newtheorem{lem}[subsubsection]{Lemma}
\newtheorem{prop}[subsubsection]{Proposition}
\newtheorem{thm}[subsubsection]{Theorem}
\newtheorem{cor}[subsubsection]{Corollary}
\theorembodyfont{\upshape}

\newtheorem{ex}[subsubsection]{Example}
\newtheorem{exs}[subsubsection]{Examples}
\newtheorem{rem}[subsubsection]{Remark}
\newtheorem{rems}[subsubsection]{Remarks}

\titleformat{\subsection}[runin]{\normalfont\normalsize\bfseries}{\thesubsection}{0.4em}{}[.]
\titleformat{\subsubsection}[runin]{\normalfont\normalsize\bfseries}{\thesubsubsection}{0.4em}{}[.]
\titleformat{\paragraph}[runin] {\normalfont\normalsize\bfseries}{\theparagraph}{0.4em}{}[.]
\titleformat{\subparagraph}[runin] {\normalfont\normalsize\bfseries}{\thesubparagraph}{0.4em}{}[.]

\titlespacing*{\subsection} {0pt}{3.25ex plus 1ex minus .2ex}{0.4em} 
\titlespacing*{\subsubsection}{0pt}{3.25ex plus 1ex minus .2ex}{0.4em} 
\titlespacing*{\paragraph} {0pt}{3.25ex plus 1ex minus .2ex}{0.4em} 
\titlespacing*{\subparagraph} {\parindent}{3.25ex plus 1ex minus .2ex}{0.4em}

\newenvironment{examples}[1][\mbox{}]
{\begin{exs}#1\vspace{-.3em}\begin{enumerate}[label=(\arabic{*})]}{\end{enumerate}\end{exs}}

\newcommand{\krelto}{\mathrel{\mathmakebox[\widthof{$\xrightarrow{\rule{1.45ex}{0ex}}$}]
{\xrightharpoonup{\rule{1.45ex}{0ex}}\hspace*{-2.4ex}{\mapstochar}\hspace*{1.8ex}}}} %Kleisli relation arrow
\renewcommand{\mapsto}{\xmapsto{\rule{1.45ex}{0ex}}}
\newcommand{\relto}{\mathrel{\mathmakebox[\widthof{$\xrightarrow{\rule{1.45ex}{0ex}}$}]
{\xrightarrow{\rule{1.45ex}{0ex}}\hspace*{-2.4ex}{\mapstochar}\hspace*{1.8ex}}}} %relation arrow
\renewcommand{\to}{\xrightarrow{\rule{1.45ex}{0ex}}} %to arrow

\DeclareMathOperator{\conv}{\mathrm{conv}} %convergence operation used for Kleisli extension
\DeclareMathOperator{\nbhd}{\mathrm{nbhd}} %neighborhood operation used for Kleisli extension

\newcommand{\co}{\mathrm{co}} %exponent for co-category (i.e. same category but reversed order)
\newcommand{\comma}{\!\downarrow\!} %down-arrow for comma-categories
\newcommand{\dand}{\qquad\text{and}\qquad} %"display" \and command
\newcommand{\df}{\emph} %used when defining a term

\renewcommand{\epsilon}{\varepsilon}
\newcommand{\ev}{\mathrm{ev}} %evalutation

\newcommand{\inv}{{-1}} %exponent for inverse
\newcommand{\kleisli}{\circ} %kleisli composition
\newcommand{\mate}[1]{#1^\flat}
\newcommand{\op}{\mathrm{op}} %exponent for opposite category
\newcommand{\yonmult}{\mathbf{m}} %Yoneda multiplication
\newcommand{\yoneda}{\mathbf{y}} %Yoneda embedding functor
\newcommand{\Yoneda}{\hspace*{0.2ex}\mathbf{Y}} %Large Yoneda embedding functor
\newcommand{\Vee}{\bigvee\nolimits}
\newcommand{\Wedge}{\bigwedge}

\newcommand{\two}{\mathsf{2}} %quantale 2
\newcommand{\I}{\mathsf{I}} %quantale I
\newcommand{\V}{V} % a quantale
\newcommand{\W}{W} % another quantale
\newcommand{\Pp}{{\mathsf{P}_{\!+}}} %positive (extended) reals with "plus" as tensor
\newcommand{\PV}{{P_\V}} % the V-powerset functor
\newcommand{\Vhomr}{\multimapdot} %for hom of V
\newcommand{\Vhoml}{\multimapdotinv} %for hom of V

\newcommand{\cat}[1]{\mathsf{#1}}

\newcommand{\App}{\cat{App}} %category of approach spaces
\newcommand{\Cats}[1]{#1\text{-}\cat{Cat}} %category of small #1-categories
\newcommand{\LaxExts}[1]{#1\text{-}\cat{LaxExt}} %category of associative lax extensions of monads to #1-Rel
\newcommand{\Mods}[1]{#1\text{-}\cat{Mod}} %category of #1-modules between #1-categories
\newcommand{\Mons}[1]{#1\text{-}\cat{Mon}} %category of #1-monoids
\newcommand{\NRel}{\cat{NRel}} %category of numerical relations
\newcommand{\Ord}{\cat{Ord}} %category of ordered sets
\newcommand{\Rel}{\cat{Rel}} %category of relations
\newcommand{\Rels}[1]{#1\text{-}\cat{Rel}} %category of #1-relations
\newcommand{\Set}{\cat{Set}} %category of sets
\newcommand{\Mnds}[1]{\cat{Mnd}_{#1}} %category of #1-structured monads
\newcommand{\Sup}{\cat{Sup}} %category of sup-semilattices
\newcommand{\Top}{\cat{Top}} %category of topological spaces
\newcommand{\URels}[1]{#1\text{-}\cat{URel}} %category of #1-relations

\newcommand{\PTV}{{\Pi}} %discrete presheaf functor

\newcommand{\mon}[1]{\mathbb{#1}}

\newcommand{\mbeta}{\mon{\bbbeta}} %ultrafilter monad
\newcommand{\mF}{\mon{F}} %filter monad
\newcommand{\mId}{\mon{I}} %identity monad
\newcommand{\mP}{\mon{P}} %powerset monad
\newcommand{\mPTV}{{\mon{\Pi}}} %discrete presheaf monad
\newcommand{\mPV}{{\mon{P}_\V}} %V-powerset monad
\newcommand{\mPW}{{\mon{P}_\W}} %V-powerset monad
\newcommand{\mR}{\mon{R}} %generic monad
\newcommand{\mS}{\mon{S}} %generic monad
\newcommand{\mT}{\mon{T}} %generic monad
\newcommand{\mU}{\mon{U}} %up-set monad

\newcommand{\ext}[1]{{\hat{#1}}}

\newcommand{\eR}{\ext{R}} %generic functor
\newcommand{\eS}{\ext{S}} %generic functor
\newcommand{\eT}{\ext{T}} %generic functor
\newcommand{\eU}{\ext{U}} %ultrafilter functor

\newcommand{\emId}{{\ext{\mId}}} %extension of identity monad
\newcommand{\emS}{{\ext{\mS}}} %extension of generic monad
\newcommand{\emT}{{\ext{\mT}}} %extension of generic monad
\newcommand{\kext}[1]{{\check{#1}}}
\newcommand{\keS}{\kext{S}} %generic functor
\newcommand{\keT}{\kext{T}} %generic functor
\newcommand{\kePTV}{{\kext{\Pi}}\hspace*{0.2ex}} %descrete presheaf functor

\newcommand{\kemS}{{\kext{\mS}}} %Kleisli extension of generic monad
\newcommand{\kemT}{{\kext{\mT}}} %Kleisli extension of generic monad
\newcommand{\kemPTV}{{\kext{\mPTV}}} %Kleisli extension of Pi monad

\newcommand{\TV}{(\mT,\V)}
\newcommand{\kTV}{(\mT,\V)}
\newcommand{\SV}{(\mS,\V)}

\newcommand{\mfrak}[1]{\mathpzc{#1}}

\newcommand{\fx}{\mfrak{x}\hspace*{1pt}}

\newcommand{\fy}{\hspace*{0.5pt}\mfrak{y}}
\newcommand{\fY}{\mfrak{Y}}
\newcommand{\fz}{\mfrak{z}}
\newcommand{\fZ}{\mfrak{Z}}

\newcommand{\bembeta}{{\overline{\mbeta}}} %Barr extension of the ultrafilter monad


\begin{document}

\title{A cottage industry of lax extensions}
%\title{Lax extensions of monads to categories of relations}

\author{Dirk Hofmann\footnote{Partial financial assistance by Portuguese funds through CIDMA (Center for Research and Development in Mathematics and Applications), and the Portuguese Foundation for Science and Technology (``FCT -- Funda\c{c}\~ao para a Ci\^encia e a Tecnologia''), within the project PEst-OE/MAT/UI4106/2014, and by the project NASONI under the contract PTDC/EEI-CTP/2341/2012 is gratefully acknowledged.}\ \ and Gavin J. Seal}

\date{\today}

%\address{Institute of Geometry, Algebra and Topology\\Ecole Polytechnique F\'ed\'erale de Lausanne\\1015 Switzerland\\\texttt{gseal@fastmail.fm}}

%Mathematics Subject Classification: 18C20, 18D10, 18D35

\maketitle

\begin{abstract}
In this work, we describe an adjunction between the comma category of $\Set$-based monads under the $\V$-powerset monad and the category of associative lax extensions of $\Set$-based monads to the category of $\V$-relations. In the process, we give a general construction of the Kleisli extension of a monad to the category of $\V$-relations.
\end{abstract}

%\textbf{Keywords:} bimorphism, tensor, monad, monoidal category, action

%%%%%%%%%%%%%%%%%%%%%%%%%%%%%%%%%%%%%%%%%%%%%%%%%%%%%%%%%%%%%%%%%%%%%%%%%%%%%%%%%%%%%%%%%%%%%%%%%%%%%%%%
\setcounter{section}{-1}
\section{Introduction}

In \cite{Bar:70}, Barr introduced the concept of a lax extension of a monad by lifting the ultrafilter monad $\mbeta$ on $\Set$ to a lax monad $\bembeta$ on the category $\Rel$ of relations. The category of lax algebras for $\bembeta$ were shown to form the category $\Top$ of topological spaces and continuous maps. This concept was developed in several directions, two of which provide the basis for the current work: in~\cite{CleHof:03} and~\cite{CleTho:03}, Clementino, Hofmann and Tholen allowed for lax extensions to the category $\Rels{\V}$ of relations valued in a quantale $\V$, and in~\cite{Sea:09}, Seal proposed an alternate construction of a lax extension, that, when applied to the filter monad $\mF$, returned $\Top$ again as category of lax algebras. In particular, contrarily to the one-to-one correspondence between monads and categories of algebras, different monads --- equipped with appropriate lax extensions --- can yield isomorphic categories of lax algebras.

The adjunction presented here as Theorem~\ref{thm:TheAdjunction} sheds some light on this situation by having a family of monads and their lax extensions correspond to an essentially unique representative. A central ingredient of this adjunction is the \emph{discrete presheaf monad} $\mPV$ on $\Set$ (also called the \emph{$\V$-powerset monad}) that extends the powerset monad $\mP=\mP_\two$ from the base $\two$ to a general quantale $\V$. Indeed, the Kleisli category of this monad is the category $\Rels{\V}$ of $\V$-relations --- whose opposite category is used to laxly extend $\Set$-based monads via the basic adjunction
\[
(-)^\circ\dashv\Rels{\V}(-,1):\Rels{\V}^\op\to\Set
\]
(see Subsections \ref{ssec:MapsV-Rel}, \ref{ssec:PV}, \ref{LaxExt} and Example~\ref{ex:V-PowersetAsPresheaf}). The category of ``neighborhood monads'', that is, of monads that play the role of the filter monad in our general context, is the comma category $(\mP_\V\comma\Mnds{\Set})$ of monads on $\Set$ under $\mPV$. The category $\LaxExts{\V}$ of ``convergence monads'', that is, of monads that convey notions of convergence, similarly to the ultrafilter monad mentioned above, has as its objects monads $\mT$ on $\Set$ equipped with an associative lax extension $\emT$ to $\Rels{\V}$. The comma category $(\mP_\V\comma\Mnds{\Set})$ then appears as a full reflective subcategory of $\LaxExts{\V}$ via the adjunction
\[
F\dashv G:(\mP_\V\comma\Mnds{\Set})\to\LaxExts{\V}
\]
of Theorem~\ref{thm:TheAdjunction}. The reflector $F$ takes a pair $(\mT,\emT)$ to the monad induced by the adjunction
\[
(-)^\sharp\dashv\URels{(\mT,\V)}(-,1):\URels{(\mT,\V)}^\op\to\Set,
\]
where $\URels{(\mT,\V)}$ is the category of \emph{unitary $(\mT,\V)$-relations} associated to the associative lax extension $\emT$ of $\mT$ to $\Rels{\V}$ (Proposition~\ref{prop:URelSetAdjun}). The embedding functor $G$ sends a monad morphism $\mPV\to\mT$ to the pair $(\mT,\kemT)$, where $\kemT$ is the Kleisli extension of $\mT$ to $\Rels{\V}$ (Subsection~\ref{Kleisli_ext}); this lax extension generalizes the construction described in \cite{Sea:09} for the $\V=\two$ case.

Let us make two remarks on the context of our article. First, in \cite{LowVro:08}, Lowen and Vroegrijk show that the category $\App$ of approach spaces and non-expansive maps can be presented as lax algebras for a ``prime functional  ideal'' monad laxly extended to $\Rel=\Rels{\two}$, contrasting the result of \cite{CleHof:03}, where the ultrafilter monad with a lax extension to $\Rels{\V}$ was used to obtain a lax-algebraic description of $\App$ (with $\V=\Pp$ the extended non-negative real half-line). These results, generalized in \cite[Corollary~IV.3.2.3]{Cle/al:14}, stem from a different point of view as the one taken here, as they allow for a ``change-base-base'', namely passing from $\V=\Pp$ to $\V=\two$ by modifying the underlying monad, whereas the current work fixes the quantale $\V$. Second, the study of lax extensions is rooted in topology and order theory, but current trends show that it is in no way limited to these fields, as the cited references show, but also \cite{EnqSac:14} and \cite{MarVen:12}.

In this article, we use the conventions and terminology of \cite{Cle/al:14} to which the reader is referred in case of doubt. In particular, we use the term ``order'' for what is elsewhere known as ``preorder'' (that is, a reflexive and transitive relation, not necessarily symmetric). Because our categories may be variously enriched, we use the term ``ordered category'' and ``monotone functor'' rather than their ``2-'' counterparts to designate the categories and functors enriched in $\Ord$, the category of ordered sets and monotone maps.

%%%%%%%%%%%%%%%%%%%%%%%%%%%%%%%%%%%%%%%%%%%%%%%%%%%%%%%

\section{The $\V$-powerset monad}\label{sec:Homs}

\subsection{Quantales} A \df{quantale} $\V=(\V,\otimes,k)$ (more precisely, a \df{unital quantale}) is a complete lattice equipped with an associative binary operation $\otimes$, its \df{tensor}, that preserves suprema in each variable:
\[
a\otimes\Vee_{\!i\in I} b_i=\Vee_{\!i\in I}(a\otimes b_i)\dand\Vee_{\!i\in I} a_i\otimes b=\Vee_{\!i\in I}(a_i\otimes b)
\]
for all $a,a_i,b,b_i\in\V$ ($i\in I$), and has a neutral element $k$. A \df{lax homomorphism} of quantales $f:(\W,\otimes,l)\to(\V,\otimes,k)$ is a monotone map $f:\W\to\V$ satisfying
\[
f(a)\otimes f(b)\le f(a\otimes b)\dand k\le f(l)
\]
for all $a,b\in\V$. A map $f:\W\to\V$ is a \df{homomorphism} of quantales if it is a sup-map and preserves the tensor and neutral elements (that is, $f(a)\otimes f(b)=f(a\otimes b)$ and $k=f(l)$).

Given $a\in\V$, the sup-map $(-)\otimes a$ is left adjoint to an inf-map $((-)\Vhoml a):\V\to\V$ that is uniquely determined by
\[
(v\otimes a)\le b\iff v\le (b\Vhoml a)
\]
for all $v,b\in\V$; hence, $(b\Vhoml a)=\Vee\{v\in\V\mid v\otimes a\le b\}$. Symmetrically, the sup-map $a\otimes(-)$ is left adjoint to an inf-map $(a\Vhomr (-)):\V\to\V$ that is uniquely determined by
\[
(a\otimes v)\le b\iff v\le (a\Vhomr b).
\]
In the case where the tensor of $\V$ is commutative, $((-)\Vhoml a)$ and $(a\Vhomr(-))$ coincide, and either of the two notations may be used.

\begin{examples}\label{ex:quantales}
\item\label{ex:quantales:two} The two-element chain $\two=\{0,1\}$ with $\otimes$ the binary infimum, and
$k=\top$ is a quantale. Here, $(b\Vhoml a)$ is the Boolean truth value of the statement $a\le b$. For any other quantale $\V=(V,\otimes,k)$, there is a canonical homomorphism $\two\to\V$ that necessarily sends $0$ to the bottom element $\bot$ of $\V$, and $1$ to $k$.
\item The unit interval $\I=[0,1]$ with its natural order, $\otimes$ given by multiplication, and neutral element $k=1$ is a quantale. In this case,
\[
(b\Vhoml a)=b\oslash a:=\Vee\{v\in[0,1]\mid v\cdot a\le b\},
\]
so that $b\oslash a=b/a$ if $b<a$, and $b\oslash a=1$ otherwise. This quantale is isomorphic to the \df{extended non-negative reals} quantale $\Pp=([0,\infty]^\co,+,0)$ via the monotone bijection $f:[0,1]\to[0,\infty]^{\co}$ that sends $0$ to $\infty$, and $x\in(0,1]$ to $-\log(x)$ (see for example \cite{CleHofTho:04}; here $[0,\infty]^\co$ denotes the set of extended non-negative reals equipped with the order opposite to the natural order). The embedding  $\two\to\I$ is a homorphism of quantales (that is, a sup-map that preserves the tensor and neutral element).
\end{examples}

\subsection{$\V$-relations}
Given a quantale $\V=(\V,\otimes,k)$, a \df{$\V$-relation} $r:X\relto Y$ from a set $X$ to a set $Y$ is a $\Set$-map $r:X\times Y\to\V$. A $\V$-relation $r:X\relto Y$ can be composed with a $\V$-relation $s:Y\relto Z$ via
\[
(s\cdot r)(x,z) = \Vee_{\!y\in Y} r(x,y)\otimes s(y,z)
\]
(for all $x\in X$, $z\in Z$) to yield a $\V$-relation $s\cdot r:X\relto Z$. The identity on a set $X$ for this composition is the $\V$-relation $1_X:X\relto X$ that sends $(x,y)$ to $k$ if $x=y$ and to $\bot$ otherwise (where $\bot$ is the least element of $\V$). Sets and $\V$-relations between them form the category
\[
\Rels{\V}.
\]

\begin{examples}
\item The category $\Rels{\two}$ is isomorphic to the category $\Rel$ of sets with ordinary relations as morphisms.
\item The category $\Rels{\I}$ is isomorphic to the category $\NRel=\Rels{\Pp}$ of numerical relations. 
\end{examples}

\subsection{Ordered $\V$-relations}\label{ssec:OrdV-Rel}
The hom-set $\Rels{\V}(X,Y)$ inherits the pointwise order induced by $\V$: given $r:X\relto Y$ and $r':X\relto Y$, we have
\[
r\le r'\iff\forall (x,y)\in X\times Y\ (r(x,y)\le r'(x,y)).
\]
Since the order on $\V$ is complete, so is the pointwise order on $\Rels{\V}(X,Y)$, and since the tensor in $\V$ distributes over suprema, $\V$-relational composition preserves suprema in each variable:
\[
\big(\Vee_{\!i\in I}s_i\big)\cdot r=\Vee_{\!i\in I}(s_i\cdot r)\dand t\cdot\big(\Vee_{\!i\in I}s_i\big)=\Vee_{\!i\in I}(t\cdot s_i)
\]
for $\V$-relations $r:X\relto Y$, $s_i:Y\relto Z$ ($i\in I)$, and $t:Z\relto W$. In particular, given a $\V$-relation $r:X\relto Y$, the sup-map $(-)\cdot r:\Rels{\V}(Y,Z)\to\Rels{\V}(X,Z)$ is left adjoint to an inf-map $(-)\Vhoml r:\Rels{\V}(X,Z)\to\Rels{\V}(Y,Z)$ defined by
\[
(s\Vhoml r)(y,z)=\Wedge_{x\in X} (r(x,y) \Vhomr s(x,z)),
\]
for all $\V$-relations $s:X\relto Z$, and $y\in Y$, $z\in Z$. Diagrammatically, $s\Vhoml r$ is the extension in $\Rels{\V}$ of $s$ along $r$:
\[
\xymatrix{X\ar[r]^-{\rule[-1pt]{0pt}{0pt}r}\ar@{}[r]|-{\object@{|}}\ar[d]|-{\object@{|}}_{s\rule{2pt}{0pt}}="a"&Y\\
Z\ar@{<--}[ur]|-{\object@{|}}_{s\,\Vhoml\, r}="b"&\ar@{}^-{\ge}"a";"b"}
\]
Symmetrically, for a $\V$-relation $t:Z\relto W$, the sup-map $t\cdot(-):\Rels{\V}(Y,Z)\to\Rels{\V}(Y,W)$ is left adjoint to an inf-map $(t\Vhomr(-)):\Rels{\V}(Y,W)\to\Rels{\V}(Y,Z)$ given by
\[
(t\Vhomr s)(y,z)=\Wedge_{w\in W} (s(y,w)\Vhoml t(z,w)),
\]
for all $\V$-relations $s:Y\relto W$, and $y\in Y$, $z\in Z$. Diagrammatically, $t\Vhomr s$ is the lifting in $\Rels{\V}$ of $s$ along $t$:
\[
%\xymatrix{Y\ar[r]^-{\rule[-2pt]{0pt}{0pt}s}\ar@{}[r]|-{\object@{|}}\ar@{-->}[dr]|-{\object@{|}}_{t\Vhomr s}&W\\
%&Z\ar[u]|-{\object@{|}}_{\rule{1pt}{0pt}t}}
\xymatrix{Y\ar[d]|-{\object@{|}}_{s\rule{2pt}{0pt}}="a"\ar@{-->}[dr]|-{\object@{|}}^{t\,\Vhomr\, s}="b"&\\
W\ar@{}[r]|-{\object@{|}}\ar@{<-}[r]_-{\rule{0pt}{6pt}t}&Z\ar@{}_-{\ge}"a";"b"}
\]

\subsection{Opposite $\V$-relations}
For a $\V$-relation $r:X\relto Y$, the \df{opposite} $\V$-relation is $r^\circ:Y\relto X$ is defined by
\[
r^\circ(y,x)=r(x,y)
\]
for all $x\in X$, $y\in Y$. Involution preserves the order on the hom-sets:
\[
r\le s\implies r^\circ\le s^\circ
\]
(for all $\V$-relations $r,s:X\relto Y$).

\subsection{Maps as $\V$-relations}\label{ssec:MapsV-Rel}
There is a functor from $\Set$ to $\Rels{\V}$ that interprets the graph of a $\Set$-map $f: X \to Y$ as the 
$\V$-relation $f_\circ:X\relto Y$ given by
\[
f_\circ(x,y)=
\begin{cases}k&\text{if $f(x)=y$,}\\ \bot&\text{otherwise.}\end{cases}
\]
To keep notations simple, we write $f:X\to Y$ instead of $f_\circ:X\relto Y$ to designate a $\V$-relation induced by a map; in particular, we write $f^\circ$ instead of $(f_\circ)^\circ$. There is therefore a functor $(-)_\circ:\Set\to\Rels{\V}$, and a functor $(-)^\circ:\Set\to\Rels{\V}^\op$ obtained by sending a map $f$ to $f_\circ=(f_\circ)^\circ$.

Note that without any commutativity assumption on $\V$, composition of $\V$-relations is not necessarily compatible with the involution $(-)^\circ$, but whiskering by $\Set$-maps is:
\[
(h\cdot s\cdot f)^\circ=f^\circ\cdot s^\circ\cdot h^\circ
\]
for all $\V$-relations $s:Y\to Z$, and $\Set$-maps $f:X\to Y$, $h:Z\to W$.

\subsection{The $\V$-powerset monad}\label{ssec:PV}
Any $\V$-relation $r:X\relto Y$ yields a map $r^\mPV:\V^Y\to\V^X$ that sends $s:Y\to\V$ to $r^\mPV(s):X\to\V$ defined by
\[
r^\mPV(s)(x)=\bigvee_{y\in Y}r(x,y)\otimes s(y)
\]
for all $x\in X$. This correspondence describes a functor $\Rels{\V}^\op\to\Set$ that is right adjoint to $(-)^\circ:\Set\to\Rels{\V}^\op$, and determines the \df{$\V$-powerset monad} $\mPV=(P_\V,\mu,\eta)$ on $\Set$, where
\begin{alignat*}{3}
\PV X&=\V^X,& \PV f(r)(y)&=\Vee_{x\in f^\inv(y)}r(x),\\
\mu_X(R)(x)&=\Vee_{s\in\V^X}s(x)\otimes R(s),\qquad &  \eta_X(x)(x')&=1_X(x,x'),
\end{alignat*}
for all $x,x'\in X$, $y\in Y$, and maps $f:X\to Y$, $r:X\to\V$, $R:\V^X\to\V$. The $\two$-monad $\mP_\two$ is easily seen to be isomorphic to the powerset monad $\mP$.

The monad associated to the left adjoint functor $(-)_\circ:\Set\to\Rels{\V}$ is of course isomorphic to $\mPV$, but our focus in the rest of this article will be on the adjunction with $\Rels{\V}^\op$ (see~\ref{ssec:DiscretePresheafMonad}).

\subsection{$\V$-categories}\label{ssec:SmallV-Cats}
When $\V=(\V,\otimes,k)$ is a quantale, a small $\V$-category $(X,a)$ is a set $X$ with a transitive and reflexive $\V$-relation $a$, so that
\[
a\cdot a\le a\dand 1_X\le a,
\]
or equivalently,
\[
a(x,y)\otimes a(y,z)\le a(x,z)\dand  k\le a(x,x),
\]
for all $x,y,z\in X$. A $\V$-functor $f:(X,a)\to(Y,b)$ of $\V$-categories is a map $f:X\to Y$ with
\[
a\cdot f^\circ\le f^\circ\cdot b,
\]
or, in pointwise notation,
\[
a(x,y)\le b(f(x),f(y))
\]
for all $x,y,z\in X$. We denote by
\[
\Cats{\V}
\]
the category of small $\V$-categories and their $\V$-functors. The induced order on a $\V$-category $(X,a)$ is defined by
\[
x\le y\iff k\le a(x,y)
\]
(for all $x,y\in X$); with the hom-sets $\Cats{V}(X,Y)$ equipped with the induced pointwise order:
\[
f\le g\iff\forall x\in X\ (f(x)\le g(x))
\]
(for all $\V$-functors $f,g:X\to Y$), $\Cats{V}$ is an ordered category.

\subsection{$\V$-modules}
Let $\V$ be a quantale, and $(X,a)$, $(Y.b)$ small $\V$-categories. A $\V$-relation $r:X\relto Y$ is a $\V$-module if
\[
r\cdot a\le r\dand b\cdot r\le r.
\]
Since the reverse inequalities always hold, these are in fact equalities:
\[
r\cdot a= r\dand b\cdot r= r.
\]
%We write $r:(X,a)\modto(Y,b)$ if the $\V$-relation $r$ is a $\V$-module.
Modules compose as $\V$-relations, and $a:(X,a)\relto (X,a)$ serves as an identity morphism in the category
\[
\Mods{\V}
\]
whose objects are $\V$-categories and morphisms are $\V$-modules. This category is ordered, with the order inherited from $\Rels{\V}$; in fact, $\Mods{\V}$ is a quantaloid, with suprema in its hom-sets formed as in $\Rels{\V}$; similarly, adjoints to the composition maps (as described in~\ref{ssec:OrdV-Rel}). By sending a $\V$-module to its underlying $\V$-relation, one obtains a lax functor
\[
\Mods{\V}\to\Rels{\V}
\]
that preserves composition, but in general only preserves the identity laxly since $1_X\le a$ for any $\V$-category structure $a:X\relto X$.

\subsection{Maps in $\Mods{\V}$}\label{ssec:V-CatopV-Mod}
For a $\V$-functor $f:(X,a)\to(Y,b)$, one defines a $\V$-module $f^\ast:(Y,b)\relto(X,a)$
by
\[
f^\ast:=f^\circ\cdot b,
\]
that is, $f^\ast(y,x)=b(y,f(x))$ for all $x\in X$, $y\in Y$. This operation defines a monotone functor
\[
(-)^\ast:\Cats{\V}\to\Mods{\V}^\op.
\]
In particular, $1_{(X,a)}=a=1_X^\ast$.

\section{The $(\mT,\V)$-powerset monad}

\subsection{Associative lax extensions}\label{LaxExt}
A \df{lax extension} of a monad $\mT=(T,m,e)$ on $\Set$ to $\Rels{\V}$ is a lax functor $\eT:\Rels{\V}\to\Rels{\V}$ that extends $T$ laxly, and such that $m^\circ:\eT\to\eT\eT$ and $e^\circ:1_{\Rels{\V}}\to\eT$ are lax natural transformations. Equivalently, $\eT:\Rels{\V}\to\Rels{\V}$ is given by functions
\[
\eT_{X,Y}:\Rels{\V}(X,Y)\to\Rels{\V}(TX,TY)
\]
for all sets $X,Y$ (with $\eT_{X,Y}$ simply written as $\eT$), such that
\begin{multicols}{2}
\begin{enumerate}
\item\label{cond:LaxExt1} $r\le r'\implies \eT r\le\eT r'$,
\item\label{cond:LaxExt2} $(Tf)^\circ\le\eT(f^\circ)$,
\item\label{cond:LaxExt4} $\eT(f^\circ\cdot r)=(Tf)^\circ\cdot\eT r$,
\item\label{cond:LaxExt5} $\eT s\cdot\eT r\le\eT(s\cdot r)$,
\item\label{cond:LaxExt6} $\eT\eT r\cdot m_X^\circ\le m_Y^\circ\cdot\eT r$,
\item\label{cond:LaxExt3} $r\cdot e_X^\circ\le e_Y^\circ\cdot\eT r$,
\end{enumerate}
\end{multicols}
for all sets $X,Y,Z$, $\V$-relations $r,r':X\relto Y$, $s:Y\relto Z$, and maps $f:Z\to Y$. The lax extension is \df{associative}, if the inequalities in conditions~\eqref{cond:LaxExt5} and \eqref{cond:LaxExt6} are equalities, that is, if $\eT$ preserves composition of $\V$-relations and $m^\circ:\eT\eT\relto\eT$ is a natural transformation.

\begin{examples}
\item An associative lax extension $\emId$ of the identity monad $\mId$ on $\Set$ to $\Rels{\V}$ is given by the identity monad on $\Rels{\V}$.
\item An associative lax extension of the ultrafilter monad $\mU$ on $\Set$ to $\Rels{\V}$ is given by the lax functor $\eU:\Rels{\V}\to\Rels{\V}$ defined by
\[
\eU r(\fx,\fy)=\Wedge_{A\in\fx,B\in\fy}\Vee_{x\in A,y\in B}r(x,y)
\]
for all ultrafilters $\fx\in UX$ and $\fy\in UY$.
\end{examples}

\begin{rem}
The definition of a \df{lax extension} $\emT$ of $\mT$ to $\Rels{\V}$ given above is equivalent to the one given in \cite[Section~III.1]{Cle/al:14}. In particular, $\eT$ is a lax functor, and $e^\circ:1_{\Rels{\V}}\to\eT$ and $m^\circ:\eT\to\eT\eT$ are lax natural transformations. The lax extension is called \emph{associative} because the equalities in \eqref{cond:LaxExt5} and \eqref{cond:LaxExt6} are equivalent to the Kleisli convolution being associative on unitary relations (see Proposition~\ref{prop:associative} below).
\end{rem}

\subsection{$\TV$-relations}
For a quantale $\V$ and an associative lax extension $\emT$ to $\Rels{\V}$ of a monad $\mT$ on $\Set$, a \df{$\TV$-relation} $r:X\krelto Y$ is a $\V$-relation $r:TX\relto Y$. The \df{Kleisli convolution} $s\kleisli r:X\krelto Z$ of $\TV$-relations $r:X\krelto Y$ and $s:Y\krelto Z$ is the $\TV$-relation defined by
\[
s\kleisli r:=s\cdot \eT r\cdot m_X^\circ.
\]
A $\TV$-relation $r:X\relto Y$ is \df{unitary} if 
\[
r\kleisli 1_X^\sharp=r\dand 1_Y^\sharp\kleisli r=r,
\]
where $1_X^\sharp:=e_X^\circ\cdot\eT 1_X:TX\relto X$; these conditions are equivalent to
\[
r\cdot \eT 1_X=r\dand e_Y^\circ\cdot\eT r\cdot m_X^\circ=r.
\]
With conditions \ref{LaxExt}\eqref{cond:LaxExt2} and \eqref{cond:LaxExt3}, the $\TV$-relation $r:X\relto Y$ is unitary precisely if
\[
r\cdot \eT 1_X\le r\dand e_Y^\circ\cdot\eT r\cdot m_X^\circ\le r.
\]
One readily verifies that $1_X^\sharp$ is a unitary $\TV$-relation. Furthermore, Kleisli convolution is associative on unitary $\TV$-relations, as the following result recalls (see~\cite[Proposition~III.1.9.4]{Cle/al:14}).
\begin{prop}\label{prop:associative}
Let $\emT$ be a lax extension to $\Rels{\V}$ of a monad $\mT=(T,m,e)$ on $\Set$. The following are equivalent:
\begin{enumerate}[label=\normalfont(\roman{*})]
\item $\emT$ is associative;
\item $\eT:\Rels{\V}\to\Rels{\V}$ preserves composition and $m^\circ:\eT\to\eT\eT$ is natural.
\end{enumerate}
\end{prop}
%In fact, one has for a unitary $\TV$-relation $r:X\krelto Y$, and \emph{any} $\TV$-relations $s:Y\krelto Z$ and $t:Z\krelto W$:
%\begin{align*}
%t\kleisli(s\kleisli r)&=t\cdot\eT(s\cdot\eT(r\cdot\eT 1_X)\cdot m_X^\circ)\cdot m_X^\circ&&\text{($r$ unitary)}\\
%&=t\cdot\eT(s\cdot\eT r\cdot\eT\eT 1_X\cdot m_X^\circ)\cdot m_X^\circ&&\text{($\eT$ preserves composition)}\\
%&=t\cdot\eT(s\cdot\eT r\cdot m_X^\circ\cdot\eT1_X)\cdot m_X^\circ&&\text{($m^\circ$ natural)}\\
%&=t\cdot\eT s\cdot\eT\eT r\cdot \eT(m_X^\circ)\cdot\eT\eT1_X\cdot m_X^\circ&&\text{($\eT$ preserves composition)}\\
%&=t\cdot\eT s\cdot\eT\eT r\cdot(T m_X)^\circ\cdot\eT\eT1_X\cdot m_X^\circ&&\text{(condition \eqref{cond:LaxExt4})}\\
%&=t\cdot\eT s\cdot\eT\eT r\cdot m_{TX}^\circ\cdot m_X^\circ\cdot\eT 1_X&&\text{($m^\circ$ natural, $m\cdot Tm=m\cdot mT$)}\\
%&=t\cdot\eT s\cdot m_Y^\circ\cdot\eT r\cdot\eT\eT 1_X\cdot m_X^\circ&&\text{($m^\circ$ natural, twice)}\\
%&=t\cdot\eT s\cdot m_Y^\circ\cdot\eT(r\cdot\eT1_X)\cdot m_X^\circ=(t\kleisli s)\kleisli r &&\text{($r$ unitary).}
%\end{align*}
%As a consequence, if $r$ and $s$ are unitary $\TV$-relations, then so is $s\kleisli r$. 

Hence, one can form the category
\[
\URels{\TV}
\]
whose objects are sets and morphisms are unitary $\TV$-relations (with composition given by Kleisli convolution and identity on $X$ by $1_X^\sharp$).

\subsection{Ordered $\TV$-relations}
Let $\V$ be a quantale and $\emT$ an associative lax extension to $\Rels{\V}$ of a monad $\mT$ on $\Set$. The category $\URels{\TV}$ forms an ordered category when the hom-sets $\URels{\TV}(X,Y)\subseteq\Rels{\V}(TX,Y)$ are equipped with the pointwise order inherited from $\Rels{\V}$, since the Kleisli convolution preserves this order on the left and right (see~\ref{LaxExt}\eqref{cond:LaxExt1}). 
 
The infimum $\Wedge_{i\in I}\varphi_i$ in $\Rels{\TV}(X,Y)$ of a family of unitary $\TV$-relations $(\varphi_i:X\krelto Y)_{i\in I}$ is again unitary:
\begin{align*}
\big(\Wedge_{i\in I}\varphi_i\big)\cdot\eT 1_X\le \Wedge_{i\in I}(\varphi_i\cdot\eT 1_X)&=\Wedge_{i\in I}\varphi_i\dand\\
e_Y^\circ\cdot\eT\big(\Wedge_{i\in I}\varphi_i\big)\cdot m_X^\circ\le\Wedge_{i\in I}(e_Y^\circ\cdot\eT\varphi_i\cdot m_X^\circ)&=\Wedge_{i\in I}\varphi_i.
\end{align*}
Hence, the ordered category $\URels{\TV}$ has complete hom-sets. Moreover, for a unitary $\TV$-relation $\varphi:X\krelto Y$, the map $(-)\kleisli\varphi:\URels{\TV}(Y,Z)\to\URels{\TV}(X,Z)$ has a right adjoint $(-)\multimapinv\varphi:\URels{\TV}(X,Z)\to\URels{\TV}(Y,Z)$ given by
\[
(\psi\multimapinv\varphi)=(\psi\Vhoml(\eT\varphi\cdot m_X^\circ)).
\]
Indeed, for all $\TV$-relations $\gamma:Y\krelto Z$ and $\psi:X\krelto Z$,
\[
 \gamma\kleisli\varphi\le\psi\iff\gamma\cdot\eT\varphi\cdot m_X^\circ\le\psi
\iff \gamma\le(\psi\Vhoml(\eT\varphi\cdot m_X^\circ)).
\]
If $\psi$ is moreover unitary, then by associativity of the Kleisli convolution,
\begin{align*}
(1_Z^\sharp\kleisli(\psi\multimapinv\varphi))\kleisli\varphi &\le 1_Z^\sharp\kleisli\psi=\psi\dand\\
((\psi\multimapinv\varphi)\kleisli 1_Y^\sharp)\kleisli\varphi &\le(\psi\multimapinv\varphi)\kleisli\varphi \le \psi;
\end{align*}
therefore, $1_Z^\sharp\kleisli(\psi\multimapinv\varphi)\le(\psi\multimapinv\varphi)$ and $(\psi\multimapinv\varphi)\kleisli 1_Y^\sharp\le(\psi\multimapinv\varphi)$, that is, $\psi\multimapinv\varphi$ is unitary. Hence, $\psi\multimapinv \varphi$ is the extension in $\URels{\TV}$ of $\psi$ along $\varphi$:
\[
%\xymatrix{X\ar@{-^`}[r]|-{\object@{|}}^-{\rule[-3pt]{0pt}{0pt}\psi}\ar@{-^`}[d]|-{\object@{|}}_{\varphi\rule{1pt}{0pt}}&Y\\
%Z\ar@{--^`}[ur]|-{\object@{|}}_{\psi\:\!\multimapinv\:\!\varphi}&}
\xymatrix{X\ar@{-^`}[r]^-{\rule[-3pt]{0pt}{0pt}\varphi}\ar@{}[r]|-{\object@{|}}\ar@{-_`}[d]|-{\object@{|}}_{\psi\rule{2pt}{0pt}}="a"&Z\\
Y\ar@{_'--}[ur]|-{\object@{|}}_{\psi\:\!\multimapinv\:\!\varphi}="b"&\ar@{}^-{\ge}"a";"b"}
\]
Hence, as it is left adjoint, the map $(-)\kleisli\varphi:\URels{\TV}(Y,Z)\to\URels{\TV}(X,Z)$ preserves suprema.

\subsection{Maps as $\TV$-relations}
Let $\V$ be a quantale and $\emT$ an associative lax extension to $\Rels{\V}$ of a monad $\mT$ on $\Set$. A map $f:X\to Y$ gives rise to a unitary $\TV$-relation $f^\sharp:Y\krelto X$ via
\[
f^\sharp:=e_X^\circ\cdot\eT(f^\circ).
\]
This definition is consistent with the notation used for the identity of the Kleisli convolution, and one also has $f^\sharp\kleisli g^\sharp=(g\cdot f)^\sharp$ for all maps $f:X\to Y$ and $g:Y\to Z$ in $\Set$. This defines a functor
\[
 (-)^\sharp:\Set\to\URels{\TV}^\op
\]
that maps objects identically. For a unitary $\TV$-relation $\varphi:X\krelto Y$, one moreover has
\[
f^\sharp\kleisli\varphi=f^\circ\cdot\varphi
\]
for all maps $f:Z\to Y$. Indeed,
\begin{align*}
f^\sharp\kleisli\varphi&=e_X^\circ\cdot\eT(f^\circ\cdot\varphi)\cdot m_X^\circ=e_X^\circ\cdot(Tf)^\circ\cdot\eT\varphi\cdot m_X^\circ\\
&=f^\circ\cdot e_Y^\circ\cdot\eT 1_Y\cdot\eT\varphi\cdot m_X^\circ=f^\circ\cdot(1_X^\sharp\kleisli\varphi)=f^\circ\cdot\varphi.
\end{align*}

\subsection{An adjunction between $\URels{\TV}^\op$ and $\Set$}\label{ssec:flatten}
We now proceed to show that the functor $(-)^\sharp:\Set\to\URels{\TV}^\op$ is left adjoint to the contravariant hom-functor
\[
 \URels{\TV}(-,1):\URels{\TV}^\op\to\Set,
\]
where $1=\{\star\}$ denotes a singleton. We identify elements $x\in X$ with maps $x:1\to X$, and to a unitary $\TV$-relation $\psi:X\krelto Y$ we associate the map $\mate{\psi}:Y\to\URels{\TV}(X,1)$ defined by
\[
\mate{\psi}(y):=y^\sharp\kleisli\psi=y^\circ\cdot\psi=\psi(-,y)
\]
for all $y\in Y$ (the third equality follows by definition of composition in $\Rels{\V}$); here, $\psi(-,y)(\fx,\star):=\psi(\fx,y)$.

\begin{lem}\label{lem:UnitIffUnitPtwise}
Let $\V$ be a quantale, $\emT$ an associative lax extension to $\Rels{\V}$ of a monad $\mT$ on $\Set$, and $\varphi:X\krelto Y$ a $\TV$-relation. Then $\varphi$ is unitary if and only if $y^\circ\cdot\varphi$ is unitary for all $y\in Y$.
\end{lem}

\begin{proof}
If $\varphi$ is unitary, then $y^\circ\cdot\varphi=y^\sharp\kleisli\varphi$ is unitary as well. To verify the other implication, suppose that $y^\circ\cdot\varphi$ is unitary for all $y\in Y$. Then one has
\[
y^\circ\cdot(\varphi\kleisli 1_X^\sharp)=(y^\circ\cdot\varphi)\kleisli1_X^\sharp=y^\circ\cdot\varphi
\]
and
\[
y^\circ\cdot(1_Y^\sharp\kleisli\varphi)
=y^\circ\cdot e_Y^\circ\cdot\eT\varphi\cdot m_X^\circ
=e_1^\circ\cdot(Ty)^\circ\cdot\eT\varphi\cdot m_X^\circ
=e_1^\circ\cdot\eT(y^\circ\cdot\varphi)\cdot m_X^\circ
=y^\circ\cdot\varphi
\]
for all $y\in Y$. Since for any $\V$-relation $r:Z\relto Y$, one has $y^\circ\cdot r(z,\star)=r(z,y)$, one can conclude that $\varphi\kleisli 1_X^\sharp=\varphi$ and $1_Y^\sharp\kleisli\varphi=\varphi$.
\end{proof}

\begin{prop}\label{prop:URelSetAdjun}
Let $\V$ be a quantale and $\emT$ an associative lax extension to $\Rels{\V}$ of a monad $\mT$ on $\Set$.
\begin{enumerate}
\item For a set $X$, the product $1^X=\prod_{x\in X}1_x$ in $\URels{\TV}$  (with $1_x=1$ for all $x\in X$) exists, and can be chosen as $1^X=X$ with projections $\pi_x=x^\sharp:X\krelto 1$ ($x\in X$). 
\item The contravariant hom-functor
\[
\URels{\TV}(-,1):\URels{\TV}^\op\to\Set
\]
has $(-)^\sharp:\Set\to\URels{\TV}^\op$ as left adjoint. The unit and counit of the associated adjunction are given by the Yoneda maps
\begin{align*}
 \yoneda_X=(1_X^\sharp)^\flat:X&\to\URels{\TV}(X,1),\quad x\mapsto x^\sharp\\
\intertext{and the evaluation $\V$-relations}
 \epsilon_X:X&\krelto\URels{\TV}(X,1),\quad\epsilon_X(\fx,\psi)=\psi(\fx,\star),
\end{align*}
respectively.
\end{enumerate}
\end{prop}

\begin{proof}\mbox{}
\begin{enumerate}
\item For a family $(\phi_x:Y\krelto 1)_{x\in X}$ of unitary $\TV$-relations, one can define a $\TV$-relation $\phi:Y\krelto X$ by setting $\phi(\fy,x)=\phi_x(\fy,\star)$ for all $\fy\in TY$. Since $x^\circ\cdot\phi=\phi_x$ is unitary for all $x\in X$, then so is $\phi$ by Lemma \ref{lem:UnitIffUnitPtwise}. Unicity of the connecting morphism $\phi:Y\krelto X$ follows from its definition. 
\item Since $1^X=\prod_{x\in X}1_x$ exists in $\URels{\TV}$ for all sets $X$, the representable functor $H=\URels{\TV}^\op(1,-):\URels{\TV}^\op\to\Set$ has a left adjoint $F=1^{(-)}$ that sends a set $X$ to its product $1^X=X$, and a map $f:X\to Y$ to the unitary $\TV$-relation $f^\sharp:Y\krelto X$. Similarly, it follows that $\yoneda$ is the unit of the adjunction and $\epsilon$ its counit. 
\end{enumerate}
\end{proof}

\subsection{The $(\mT,\V)$-powerset monad}\label{ssec:DiscretePresheafMonad}
Let $\V$ be a quantale, $\emT$ an associative lax extension to $\Rels{\V}$ of a monad $\mT$ on $\Set$. The adjunction described in Proposition~\ref{prop:URelSetAdjun} induces a monad
\[
\mPTV=(\PTV,\yonmult,\yoneda) 
\]
on $\Set$, the \df{$(\mT,\V)$-powerset monad} (or, more generically, the \df{discrete presheaf monad}) associated to $\emT$, where
\begin{alignat*}{3}
\PTV X&=\URels{\TV}(X,1),\qquad&\PTV f(\psi)&=\psi\kleisli f^\sharp,\\
\yonmult_X(\Psi)&=\Psi\kleisli\epsilon_X, &  \yoneda_X(x)&=x^\sharp,
\end{alignat*}
for all $x\in X$, $f:X\to Y$, and unitary $\TV$-relations $\psi:X\krelto 1$, $\Psi:\PTV X\krelto 1$. Let us emphasize that $\mPTV$  depends both on a choice of a quantale $\V$, as well as that of an associative lax extension $\emT$:
\[
\mPTV=\mPTV(\mT,\emT)=\mPTV(\mT,\V,\emT).
\]

\begin{ex}\label{ex:V-PowersetAsPresheaf}
Consider a quantale $\V$, and the identity monad $\mT=\mId=(1_\Set,1,1)$ on $\Set$ extended to the identity monad $\emId=(1_{\Rels{\V}},1,1)$ on $\Rels{\V}$. Unitary $\TV$-relations are simply $\V$-relations, so, denoting by $\mPV=(\PV,\mu,\eta)$ the $(\mId,\V)$-powerset monad associated to $\emId$, we compute
\begin{alignat*}{3}
\PV X&=\Rels{\V}(X,1),\qquad&\PV f(r)&=r\cdot f^\circ,\\
\mu_X(R)&=R\cdot\ev_X, &  \eta_X(x)&=x^\circ,
\end{alignat*}
for all $x\in X$, $f:X\to Y$, and $\V$-relations $r:X\relto 1$, $R:\PV X\relto 1$ (where $\ev_X:X\relto\Rels{V}(X,1)$ is given by $\ev_X(x,r)=r(x,\star)$). By further identifying $\V$-relations $X\relto 1$ with maps $X\to\V$, we obtain the $\V$-powerset monad of~\ref{ssec:PV}. However, from now on we will only use the description given here, as it allows for more concise arguments\footnote{In view of this example, $\mP_{(\mT,\V)}$ is a more logical notation for the $(\mT,\V)$-powerset monad $\mPTV$; nevertheless, we favor the latter, simpler, notation whenever $\mT$ is not the identity monad.}.
\end{ex}

\subsection{The Kleisli category of $\mPV$}\label{ssec:KleisliPV}
The  Kleisli category of the $\V$-powerset monad $\mPV$ has sets as objects, and maps $f:Y\to\Rels{\V}(X,1)$ as morphisms from $Y$ to $Z$. Kleisli composition $f\kleisli g:Z\to\Rels{\V}(X,1)$ of $g:Z\to\Rels{\V}(Y,1)$ and $f:Y\to\Rels{\V}(X,1)$ is
\[
f\kleisli g(z)=\mu_X\cdot \PV f\cdot g(z)=g(z)\cdot f^\circ\cdot\ev_X.
\]
The comparison functor $K:\Set_{\mPV}\to\Rels{\V}^\op$ sends a map $f:Y\to\Rels{\V}(X,1)$ to the $\V$-relation $f^\circ\cdot\ev_X:X\relto Y$, and has an inverse $(-)^\flat:\Rels{\V}^\op\to\Set_{\mPV}$ that sends a $\V$-relation $r:X\relto Y$ to its mate $r^\flat:Y\to\Rels{\V}(X,1)$ (see~\ref{ssec:flatten}). Indeed, one verifies that $(-)^\flat$ is a functor:
\[
r^\flat\kleisli s^\flat=(s\cdot r)^\flat,
\]
and that it is an inverse to $K$:
\[
(r^\flat)^\circ\cdot\ev_X=r\dand(f^\circ\cdot\ev_X)^\flat=f,
\] 
for all $\V$-relations $r:X\relto Y$, $s:Y\relto Z$ and maps $f:Y\to\Rels{\V}(X,1)$. When hom-sets of $\Set_\mPV$ are equipped with the pointwise order:
\[
f\le g\iff\forall x\in X\ (f(x)\le g(x))
\]
(for $f,g:Y\to\Rels{(\V,1)}$), the isomorphism $\Rels{\V}^\op\cong\Set_\mPV$ is an order-isomorphism.

%%%%%%%%%%%%%%%%%%revise from here down%%%%%%%%%%%%%%%%%%%%%%%%%%%%%%%%%%%%%%%%%%%%%%%%%%%%%%%%%%%%%%%%%%%%%%%%%%%%%%%%%%%%%%%%%%%%%%%%%%%%%%%%%%%%%%%%%%%%%%%%%%%%%%%%%%%%%%%%%%%%%%%%%%%%%%%%%%%%%%%%%%%%%%%%%%%%%%%%%%%%%%%%%%%%%%%%%%%%%%%%%%%%%%%%%%%%%%%%%%%%
\section{The enriched Kleisli extension}\label{sec:Kleisli}%%%%%%%%%%%%%%%%%%%%%%%%%%%%%

\subsection{The Eilenberg--Moore category of $\mPV$}\label{ssec:Actions}
The $\V$-powerset monad $\mPV=(\PV,\mu,\eta)$ is induced from the adjunction $(-)^\circ\dashv(\Rels{\V}(-,1)):\Rels{\V}^\op\to\Set$. By \cite{PedTho:89} and since $\Rels{V}^\op\cong\Rels{\V^\op}$, the Eilenberg--Moore category of $\mPV$ is isomorphic to the category of left $\V^\op$-actions in the category $\Sup$ of complete lattices and sup-maps; that is, $\Set^\mPV$ is the category of right $\V$-actions in $\Sup$. We now describe this correspondence in our setting.

A right $\V$-action $(-)\ast(-):X\times\V\to X$ defines a $\mPV$-algebra structure $a:\PV X\to X$ via
\[
a(\phi):=\Vee_{x\in X}x\ast\phi(x,\star)
\]
for all $\phi\in\PV X$.

For $v\in\V$, we define the $\V$-relation $\eta_X(x)\otimes v:X\relto 1$ by $(\eta_X(x)\otimes v)(y):=\eta_X(x)(y,\star)\otimes v$ (that is, $\eta_X(x)\otimes v$ sends $(x,\star)$ to $v$, and all other pairs $(y,\star)$ to $\bot$). A $\mPV$-algebra $(X,a)$ yields a right $\V$-action on $X$ in $\Sup$ via
\[
x\ast v:=a(\eta_X(x)\otimes v),
\]
(for all $x\in X$ and $v\in\V$), so that  $(-)\ast(-):X\times\V\to X$ is a map that preserves suprema in each variable and satisfies
\[
x\ast(v\otimes u)=(x\ast v)\ast u\dand x\ast k=x
\]
for all $u,v\in\V$, $x\in X$. In particular, for $\phi\in\PV X$ and $v\in\V$, the action induced on the free $\mPV$-algebra $(\PV X,\mu_X)$ is given by $\phi\ast v=\mu_X(d_{\PV X}(\phi)\otimes v)=(\eta_{\PV X}(\phi)\otimes v)\cdot\ev_X$, and one observes
\[
\phi\ast v=\phi\otimes v.
\]

From now on, we denote the category $\Set^\mPV$ of right $\V$-actions in $\Sup$ by
\[
\Sup^\V
\]
as it is isomorphic to the functor category of $\V$ into $\Sup$ (with $\V=(\V,\otimes,k)$ considered as a one-object category). 

\subsection{From $\mPV$ to $\mT$}
Certain monads $\mT$ on $\Set$ equipped with a monad morphism $\tau:\mPV\to\mT$ will allow for a particular lax extension to $\Rels{\V}$. With is in mind, we mention the following equivalences.

\begin{prop}\label{prop:PMndEquiv}
For a monad $\mT=(T,m,e)$ on $\Set$, there is a one-to-one correspondence between:
\begin{enumerate}[label=\normalfont(\roman{*})]
\item monad morphisms $\tau:\mPV\to\mT$;
\item\label{cond2:prop:PMndEquiv} extensions $E$ of the left adjoint $F_\mT:\Set\to\Set_\mT$ along the functor $(-)^\circ:\Set\to\Rels{\V}^\op$:
\[
\xymatrix{\Rels{\V}^\op\ar[r]^-{E}&\Set_\mT\\
\Set\ar[u]^-{(-)^\circ}\ar[ur]_-{F_\mT}&}
\]
\item\label{cond3:prop:PMndEquiv} liftings $L$ of the right adjoint $G_\mT:\Set_\mT\to\Set$ along the forgetful functor $\Sup^\V\to\Set$:
\[
\xymatrix{\Set_\mT\ar[r]^L\ar[dr]_{G_\mT}&\Sup^\V\ar[d]\\
&\Set}
\]
\item $\V$-actions in $\Sup$ on $TX$ such that $Tf:TX\to TY$ and $m_X:TTX\to TX$ are equivariant sup-maps for all maps $f:X\to Y$ and sets $X$.
\end{enumerate}
\end{prop}

\begin{proof}
The equivalence between (i) and (ii) is a particular instance of the one-to-one correspondence between monad morphisms and functors between Kleisli categories of the respective monads. The equivalence between (i) and (iii) follows from the explicit description of $G_\mT$ on $\Set_\mT$-morphisms $f:X\to TY$ as $G_\mT f=m_Y\cdot Tf$. (iv) is just a restatement of (iii) in which the objects and morphisms of the Eilenberg--Moore category of $\mPV$ are described explicitly.
\end{proof}

\subsection{$\V$-power-enriched monads}\label{ssec:Free}
As mentioned in \cite{PedTho:89} and \cite{Sea:10}, there is a monotone functor
\[
\Sup^\V\to\Cats{\V}
\]
that sends a right $\V$-action $(X,a)$ to the $\V$-category $(X,(a^\dashv)^\circ\cdot\ev_X)=(X,K(a^\dashv))$  (that is, $K:\Set_\mPV\to\Rels{\V}^\op$ denotes the left adjoint of $(-)^\flat$, see \ref{ssec:KleisliPV}), where $a^\dashv:X\to \PV X$ is the right adjoint retract of $a$:
\[
1_{\PV X}\le a^\dashv\cdot a\qquad\text{and}\qquad a\cdot a^\dashv=1_X,
\]
and induces the $\V$-relation $K(a^\dashv):X\relto X$:
\[
K(a^\dashv)(x,y)=\Vee\{\phi(x,\star)\mid\phi\in\Rels{\V}(X,1):a(\phi)= y\}
\]
for all $x,y\in X$. Writing $a(\phi)=\Vee_{x\in X}x\ast\phi(x,\star)$ (see~\ref{ssec:Actions}), for $v\in\V$, one observes
\[
x\ast v\le y\iff v\le K(a^\dashv)(x,y).
\]
Setting $v=K(a^\dashv)(x,y)$, one has $x\ast K(a^\dashv)(x,y)\le y$, so that with $k\le K(a^\dashv)(x,x)$, one obtains
\[
\Vee_{x\in X}x\ast K(a^\dashv)(x,y)=y
\]
for all $x,y\in X$. Moreover, a free $\mPV$-algebra $(\PV X, \mu_X)$ has an underlying $\V$-category with internal hom $K(\mu_X^\dashv):\PV X\relto \PV X$ given by
\begin{align*}
K(\mu_X^\dashv)(r,s)&=\Vee\{R(r,\star)\mid R\in\Rels{\V}(\PV X,1):t(-,\star)\otimes R(t,\star)\le s(-,\star)\text{ for all }t\in\PV X\}\\
&=\Wedge\{(r(x,\star)\Vhomr s(x,\star))\mid x\in X\}=(s\Vhoml r)
\end{align*}
for all $r,s\in\Rels{\V}(X,1)$ (see \ref{ssec:OrdV-Rel}). 

By Proposition~\ref{prop:PMndEquiv}, a morphism $\tau:\mPV\to\mT$ of monads on $\Set$ equips the underlying set $TX$ of a free $\mT$-algebra with the internal hom
\[
(\fy\Vhoml \fx):=K(m_X\cdot\tau_{TX})^\dashv(\fx,\fy)
\]
for all $\fx\in TX$, $\fy\in TY$. The sets $TX$ are therefore $\V$-categories, and are equipped with an order (see \ref{ssec:SmallV-Cats}) that is inherited pointwise by the hom-sets $\Set(X,TY)$:
\[
f\le g\iff \forall x\in X\ (f(x)\le g(x))
\]
for all $f,g:X\to TY$. 

A \df{$\V$-power-enriched monad} is a pair $(\mT,\tau)$ with $\mT$ a monad on $\Set$ and $\tau:\mPV\to\mT$ a monad morphism such that
\[
f\le g\implies Lf\le Lg,
\]
for all $f,g:X\to TY$, where $L=m\cdot T(-):\Set_\mT\to\Sup^\V$ is the lifting of $G_\mT$ described in Proposition~\ref{prop:PMndEquiv}.  If $(\mT,\tau)$ is $\V$-power-enriched, then, with the mentioned order, $\Set_\mT$ becomes an ordered category, and the functors $E:\Rels{\V}\to\Set_\mT$ and $L:\Set_\mT\to\Sup^\V$ of Proposition~\ref{prop:PMndEquiv} become monotone.

We denote by $\Mnds{\Set}$ the category of monads on $\Set$ with their morphisms. A \df{morphism of $\V$-power-enriched monads} $\alpha:(\mS,\sigma)\to(\mT,\tau)$ is a morphism in the comma category $(\mPV\comma\Mnds{\Set})$, that is, the diagram
\[
\xymatrix@C=2em{&\mPV\ar[dl]_{\sigma}\ar[dr]^{\tau}&\\
\mS\ar[rr]^-{\alpha}&&\mT}
\]
must commute in $\Mnds{\Set}$.

\subsection{Kleisli extensions of $\V$-power-enriched monads}\label{Kleisli_ext}
Let $(\mT,\tau)$ be a $\V$-power-enriched monad. By composing the functors $(-)^\flat:\Rels{\V}^\op
\to\Set_\mPV$, $E=\Set_\tau:\Set_\mPV\to\Set_\mT$ and $L:\Set_\mT\to\Sup^\V$ of \ref{ssec:KleisliPV} and Proposition~\ref{prop:PMndEquiv}, one obtains a functor
\[
\xymatrix{(-)^\tau:\Rels{\V}^\op\ar[r]^-{(-)^\flat}&\Set_\mPV\ar[r]^-{\Set_\tau}&\Set_\mT\ar[r]^-{L}&\Sup^\V}
\]
that sends a set $X$ to $TX$, and a $\V$-relation $r:X\relto Y$ to the map $r^\tau:TY\to TX$, with
\[
r^\tau:=m_X\cdot T(\tau_X \cdot r^\flat).
\]
The \df{Kleisli extension} $\keT$ of $\mT$ to $\Rels{\V}$ (with respect to $\tau$) is defined by the functions $\keT=\keT_{X,Y}:\Rels{\V}(X,Y)\to\Rels{\V}(TX,TY)$ (indexed by sets $X$ and $Y$), with
\[
\keT r(\fx,\fy)=(r^\tau(\fy)\Vhoml\fx)
\]
for all $\V$-relations $r:X\relto Y$, and $\fx \in TX$, $\fy\in TY$.

\subsection{Kleisli extensions are lax extensions}
To prove that $\keT:\Rels{\V}\to\Rels{\V}$ is indeed a lax extension of the $\Set$-functor $T$, it is convenient to express the former as a composite of lax functors. In view of this, we remark that $\keT r$ (for a relation $r:X\relto Y$) can be written equivalently as
\[
\keT r=(r^\tau)^\ast:TX\relto TY\qquad\text{or}\qquad (\keT r)^\flat=(m_X\cdot\tau_{TX})^\dashv\cdot r^\tau
\]
(where the functor  $(-)^\ast:\Cats{\V}\to\Mods{\V}^\op$ is defined in \ref{ssec:V-CatopV-Mod}, and $(m_X\cdot\tau_{TX})^\dashv$ is right adjoint to $m_X\cdot\tau_{TX}$, see \ref{ssec:Free}). The Kleisli extension is therefore an ordered functor
\[
\xymatrix{\keT:\Rels{\V}^\op\ar[r]^-{(-)^\tau}&\Sup^\V\ar[r]&\Cats{\V}\ar[r]^-{(-)^\ast}&\Mods{\V}^\op.}
\]
There is moreover a lax functor $\Mods{\V}^\op\to\Rels{\V}^\op$ that assigns to a module its underlying relation: composition of $\V$-modules is composition of $\V$-relations, identity $\V$-modules are transitive and reflexive $\V$-relations, and $1_X\le a$ for any $\V$-category $(X,a)$. Hence, the Kleisli extension $\keT^\op$ can be decomposed as the top line of the commutative diagram
\begin{equation}\label{eq:KlExtLaxExt}
\xymatrix{\Rels{\V}^\op\ar[r]^-{(-)^\flat}&\Set_\mPV\ar[r]^-{\Set_\tau}&\Set_\mT\ar[r]^-{L}\ar[dr]|-{G_\mT}&\Sup^\V\ar[r]\ar[d]&\Cats{\V}\ar[r]^-{(-)^\ast}\ar[dl]&\Mods{\V}^\op\ar[r]&\Rels{\V}^\op\\
&\Set\ar[ul]|-{(-)^\circ}\ar[u]|-{F_\mPV}\ar[ur]|-{F_\mT}\ar[rr]^-{T}&&\Set&&}
\end{equation}
in which all arrows except $\Mods{\V}^\op\to\Rels{\V}^\op$ are functors, and the latter is a lax functor that fails only to preserve identities. 

\begin{prop}\label{prop_Kleisli ext}
Given a $\V$-power-enriched monad $(\mT,\tau)$, the Kleisli extension $\keT$ of $T$ to $\Rels{\V}$ yields a lax extension $\kemT=(\keT,m,e)$ of $\mT=(T,m,e)$ to $\Rels{\V}$. Moreover, $\keT$ preserves composition of $\V$-relations.
\end{prop}

\begin{proof}
The fact that $\keT:\Rel\to\Rel$ is monotone and preserves composition of $\V$-relations follows from its decomposition as lax functors preserving composition in the first line of \eqref{eq:KlExtLaxExt}. The lax extension condition $(Tf)^\circ\le\keT(f)^\circ$ can be deduced from the diagram
\[
\xymatrix{\Rels{\V}^\op\ar[r]^-{(-)^\tau}&\Cats{\V}\ar[d]|(0.5){}="tag"\ar[r]^-{(-)^\ast}&\Mods{\V}^\op\ar[r]&\Rels{\V}^\op\ar[d]^{1_{\Rels{\V}^\op}}|(0.5){}="Set"\\
\Set\ar[u]^-{(-)^\circ}\ar[r]^{T}&\Set\ar[rr]^-{(-)^\circ}\ar@{}"tag";"Set"|(0.7){\le}&&\Rels{\V}^\op}
\]
in which the first line is $\keT^\op$. The condition $\keT(h^\circ\cdot r)=(Th)^\circ\cdot\eT r$ for all $\V$-relations $r:X\relto Y$ and maps $h:Z\to Y$ comes from
\[
\keT(h^\circ\cdot r)(\fx,\fz)=((r^\tau\cdot(h^\circ)^\tau(\fz))\Vhoml\fx)=((r^\tau\cdot Th(\fz))\Vhoml\fx)=(Th)^\circ\cdot\keT r(\fx,\fz)
\]
for all $\fx\in TX$, $\fz\in TZ$. To verify oplaxness of $e:1_\Rels{\V}\to \keT$, consider a $\V$-relation $r:X\relto Y$, and $x\in X$, $y\in Y$. Since $\tau_{X}:\PV X\to TX$ is a $\mPV$-algebra morphism, one has
\[
e_X(x)\ast r(x,y)=\tau_X(\eta_X(x)\otimes r(x,y))\le\tau_X(\Vee_{x\in X}\eta_X(x)\otimes r(x,y))=\tau_X\cdot r^\flat(y)=r^\tau\cdot e_Y(y),
\]
so that $r(x,y)\le ((r^\tau\cdot e_Y(y))\Vhoml e_X(x))=\eT r(e_X(x),e_Y(y))$, as required.

Via the isomorphism $\Rels{\V}^\op\cong\Set_\mPV$, naturality of $m^\circ:\keT\to\keT\keT$ is equivalent to $(m_X^\circ)^\flat\kleisli(\keT\keT r)^\flat=(\keT r)^\flat\kleisli (m_Y^\circ)^\flat$. As $(m_X^\circ)^\flat=d_{TX}\cdot m_X$ by commutativity of the left-most triangle in \eqref{eq:KlExtLaxExt}, naturality of $m^\circ$ is equivalent to
\[
\PV m_X\cdot(m_{TX}\cdot\tau_{TTX})^\dashv\cdot(\keT r)^\tau=(m_{X}\cdot\tau_{TX})^\dashv\cdot r^\tau\cdot m_Y.
\]
Since
\[
\PV m_X\cdot(m_{TX}\cdot\tau_{TTX})^\dashv\le(m_{X}\cdot\tau_{TX})^\dashv\cdot m_X,
\]
oplaxness follows from
\begin{align*}
m_{X}\cdot(\keT r)^\tau&=m_{X}\cdot m_{TX}\cdot T(\tau_{TX}\cdot(m_X\cdot\tau_{TX})^\dashv\cdot r^\tau)\\
&=m_{X}\cdot T(m_X\cdot\tau_{TX}\cdot(m_X\cdot\tau_{TX})^\dashv\cdot r^\tau)\\
&=m_{X}\cdot T(m_X\cdot T(\tau_X\cdot r^\flat))=r^\tau\cdot m_Y.
\end{align*}
\end{proof}

\subsection{Discrete presheaf monads of $\V$-power-enriched monads}\label{ssec:DiscretePresheafOfVEnriched}
Let $(\mT,\tau)$ be a $\V$-power-enriched monad. The monotone maps
\begin{align*}
\nbhd&=\nbhd_{X,Y}:\Rels{\V}(TX,Y)\to\Set(Y,TX),\ r\mapsto(m_X\cdot\tau_{TX})\cdot r^\flat\\
\conv&=\conv_{Y,X}:\Set(Y,TX)\to\Rels{\V}(TX,Y),\ f\mapsto ((m_X\cdot\tau_{TX})^\dashv\cdot f)^\circ\cdot\ev_{TX}
\end{align*}
form an adjunction $\nbhd\dashv\conv$ for all sets $X,Y$, such that moreover $\nbhd\cdot\conv=1_{\Set(Y,TX})$. With respect to the Kleisli extension of $\mT$, one observes that
\[
((m_X\cdot\tau_{TX})^\dashv)^\circ\cdot\ev_{TX}=\keT 1_{X},
\]
so $\conv(f)=f^\circ\cdot\keT 1_{X}$, and $\conv(f)$ is a unitary $\kTV$-relation.

\begin{lem}\label{lem:UnitaryTV-Rel}
Let $(\mT,\tau)$ be a $\V$-power-enriched monad equipped with its Kleisli extension $\keT$ to $\Rels{\V}$. If $r:TX\relto Y$ is a unitary $\kTV$-relation, then
\[
\conv\cdot\nbhd(r)=r.
\]
\end{lem}

\begin{proof}
Since $r$ is unitary,
\begin{align*}
r^\flat&\le(m_X\cdot\tau_{TX})^\dashv\cdot (m_X\cdot\tau_{TX})\cdot r^\flat\\
&=\mu_{TX}\cdot\PV((m_X\cdot\tau_{TX})^\dashv\cdot m_X)\cdot\eta_{TX}\cdot\tau_{TX}\cdot r^\flat\\
&\le\mu_{TX}\cdot\PV((m_X\cdot\tau_{TX})^\dashv\cdot m_X)\cdot(m_X\cdot\tau_{TX})^\dashv\cdot\tau_{TX}\cdot r^\flat\\
&=\mu_{TX}\cdot\PV((m_X\cdot\tau_{TX})^\dashv)\cdot(e_Y^\circ\cdot\keT r\cdot m_X^\circ)^\flat\\
&=(e_Y^\circ\cdot\keT r\cdot m_X^\circ\cdot\keT 1_X)^\flat=r^\flat.
\end{align*}
This implies $r^\flat=(m_X\cdot\tau_{TX})^\dashv\cdot (m_X\cdot\tau_{TX})\cdot r^\flat$, and thus $\conv\cdot\nbhd(r)=r$, as claimed.
\end{proof}

\begin{thm}\label{thm:ConvAndNbhd}
Let $(\mT,\tau)$ be a $\V$-power-enriched monad. The Kleisli extension $\kemT$ to $\Rels{\V}$ of $\mT$ is associative, and the maps $\nbhd$ and $\conv$ yield mutually inverse monotone functors
\[
\nbhd:\URels{\kTV}^\op\to\Set_\mT\qquad\text{and}\qquad\conv:\Set_\mT\to\URels{\kTV}^\op
\]
that commute with the left adjoint functors and $(-)^\sharp:\Set\to\URels{\kTV}^\op$ and $F_\mT:\Set\to\Set_\mT$.
\end{thm} 

\begin{proof}
By Lemma~\ref{lem:UnitaryTV-Rel} and the preceding discussion, the maps $\nbhd$ and $\conv$ yield order-isomorphisms between the set of all unitary $\V$-relations $r:TX\relto Y$ and $\Set(Y,TX)$ for all sets $X,Y$. To see that these determine functors, we first need to verify the identities
\begin{align*}
\nbhd(s\kleisli r)&=\nbhd(r)\kleisli\nbhd(s),&\conv(f)\kleisli\conv(g)&=\conv(g\kleisli f)\\
\nbhd(1_X^\sharp)&=e_X,&\conv(e_X)&=1_X^\sharp
\end{align*}
for all unitary $\V$-relations $r,s:TX\relto Y$, and maps $f,g:Y\to TX$. The fact that $\nbhd$ preserves composition can be verified by using that $\tau$ is a monad morphism, and that $m$ is the multiplication of $\mT$. Preservation of composition by $\conv$ then follows because $\nbhd$ and $\conv$ are mutual inverses. Preservation of units is immediately verified by using that $\conv(f)=f^\circ\cdot\keT 1_{TX}$. Since Kleisli composition is associative, so is Kleisli convolution of unitary $(\mT,\V)$-relations, and therefore the Kleisli extension $\keT$ is an associative lax extension by Proposition~\ref{prop:associative}. Hence, one can indeed form the category $\URels{\TV}$.

Commutativity of $\nbhd$ with $(-)^\sharp$ and $F_\mT$, follows from the fact that for any map $f:X\to Y$,
\begin{align*}
m_X\cdot\tau_{TX}\cdot(f^\sharp)^\flat&=m_X\cdot\tau_{TX}\cdot(e_Y^\circ\cdot\keT(f^\circ))^\flat\\
&=m_X\cdot\tau_{TX}\cdot\mu_Y\cdot\PV((\keT(f^\circ))^\flat)\cdot(e_X^\circ)^\flat\\
&=m_X\cdot\tau_{TX}\cdot(\keT(f^\circ))^\flat\cdot e_X\\
&=m_X\cdot\tau_{TX}\cdot(m_Y\cdot\tau_{TY})^\dashv\cdot(f^\circ)^\tau\cdot e_X\\
&=\tau_Y\cdot(f^\circ)^\flat=e_Y\cdot f,
\end{align*}
by the definitions and the fact that $(g^\circ)^\flat=\eta_Y\cdot g$ for any map $g:X\to Y$ (here, we continue to use the notation $\mPV=(\PV,\mu,\eta)$). Commutativity of $\conv$ with $(-)^\sharp$ and $F_\mT$ is then immediate because $\conv$ is inverse to $\nbhd$.  
\end{proof}

\begin{cor}\label{cor:ConvAndNbhd}
Let $(\mT,\tau)$ be a $\V$-power-enriched monad equipped with its Kleisli extension $\kemT$. The discrete presheaf monad associate to $\kemT$ is order-isomorphic to $\mT$:
\[
\mPTV(\mT,\kemT)\cong\mT.
\]
\end{cor}

\begin{proof}
By Theorem~\ref{thm:ConvAndNbhd}, the isomorphism $\Rels{(\mT,\V)}^\op\cong\Set_\mT$ commutes with the left adjoint functors from $\Set$, so the induced monads are isomorphic, with corresponding orders on the sets $\PTV X$ and $TX$.
\end{proof}

\subsection{Discrete presheaf monads are $\V$-power-enriched}
For an associative lax extension $\emT$ to $\Rels{\V}$ of a monad $\mT$ on $\Set$, there is a functor 
\[
(-)_\sharp:\Rels{V}\to\URels{\TV}
\]
that sends a $\V$-relation $r:X\relto Y$ to the unitary $\TV$-relation
\[
r_\sharp:=e_Y^\circ\cdot\eT r:X\krelto Y.
\]
%Recall that an associative lax extension $\emT$ to $\Rels{\V}$ of a monad $\mT$ on $\Set$ yields a corresponding discrete presheaf monad $\mPi(\emT)$.

\begin{prop}\label{prop:NatTransfIntoPi}
Consider the discrete presheaf monad $\mPTV$ on $\Set$ associated to an associative lax extension $\emT$ to $\Rels{\V}$ of a monad $\mT$ on $\Set$. There is a functor $\Rels{\V}^\op\to\Set_\mPTV$ and an isomorphism $Q:\Set_\mPTV\to\URels{\TV}^\op$ that make the following diagram commute
\[
\xymatrix@C=1em{&\ \ \Rels{\V}^\op\ar[dl]\ar[dr]^-{(-)_\sharp}&\\
\Set_\mPTV\ar[rr]^-{Q}&&\URels{\TV}^\op}
\]
under $\Set$. In particular, there is a monad morphism $\pi:\mPV\to\mPTV$ whose component at a set $X$ sends $\varphi\in\Rels{V}(X,1)$ to the unitary $\TV$-relation $e_1^\circ\cdot\eT\varphi:X\krelto 1$.
\end{prop}

\begin{proof}
Since the monad $\PTV$ comes from the adjunction of $\URels{\TV}^\op$ over $\Set$, the functor $Q:\Set_\mPTV\to\URels{(\mT,\V)}^\op$ is the fully faithful comparison functor from the Kleisli category. This functor is bijective on objects because the left adjoint $(-)^\sharp:\Set\to\URels{(\mT,\V)}^\op$ is so. Hence, $QX=X$ for each set $X$, and $Q$ sends a morphism $f:X\to \PTV Y$ in $\Set_\mPTV$ to the unitary $\TV$-relation $f^\sharp\kleisli\epsilon_Y:Y\krelto X$. The inverse of $Q$ sends a unitary $\TV$-relation $\varphi:Y\krelto X$ to $\mate{\varphi}:X\to\PTV Y$ (see~\ref{ssec:flatten}).

The functor $\Rels{\V}^\op\to\Set_\mPTV$ is the composite of $(-)_\sharp$ with the inverse of $Q$. Explicitly, a $\V$-relation $r:X\relto Y$ is sent to the Kleisli morphism $(e_Y^\circ\cdot\eT r)^\flat:Y\to\PTV Y$. The last statement follows from the one-to-one correspondence between monad morphisms and functors between Kleisli categories.
\end{proof}

\begin{thm}\label{thm:KleisliOfPi}
Let $\emT$ be an associative lax extension to $\Rels{\V}$ of a monad $\mT$ on $\Set$. The monad morphism $\pi:\mPV\to\mPTV$ makes $\mPTV$ into a $\V$-power-enriched monad, and the pointwise order induced on the hom-sets makes $Q:\Set_\mPTV\to\URels{(\mT,\V)}^\op$ into an order-isomorphism
\[
\Set_\mPTV\cong\URels{(\mT,\V)}^\op
\]
(see Proposition~\ref{prop:NatTransfIntoPi}).
\end{thm}

\begin{proof}
Consider the natural transformation $\pi:\mPV\to\mPTV$ of Proposition~\ref{prop:NatTransfIntoPi}, and for $\varphi,\psi\in\PTV X$, define $\chi_{\{\varphi,\psi\}}:TX\relto 1$ to be the characteristic function of $\{\phi,\psi\}$ (so that $\chi_{\{\varphi,\psi\}}(\fx,\star)$ takes value $k\in\V$ if $\fx\in\{\varphi,\psi\}$ and $\bot$ otherwise). Using that $\chi_{\{\varphi,\psi\}}\cdot\epsilon_X=\varphi\vee\psi$ (in the pointwise order of $\PTV X=\URels{\TV}(X,1)$), we compute
\[
\yonmult_X\cdot\pi_{T X}(\chi_{\{\varphi,\psi\}})=e_1^\circ\cdot\eT(\varphi\vee\psi)\cdot m_X^\circ,
\]
so in particular $\varphi\vee\psi\le\yonmult_X\cdot\pi_{T X}(\chi_{\{\varphi,\psi\}})$. Hence, $\varphi\le\psi$ if and only if $\yonmult_X\cdot\pi_{T X}(\chi_{\{\varphi,\psi\}})=\psi$. Since the order $\prec$ induced on $\PTV X$ by $\tau$ is given by
\[
\varphi\prec\psi\iff\yonmult_X\cdot\pi_{T X}(\chi_{\{\varphi,\psi\}})=\psi,
\]
the relation $\prec$ describes the pointwise order of $\PTV X$. The condition for $\mPTV$ being $\V$-power-enriched is equivalent to the requirement that $\Set_\mPTV$ form an ordered category with respect to the order induced by $\pi$; but this follows immediately from the fact that $\Rels{\V}$ is ordered. 
\end{proof}

\begin{rem}\label{rem:ActionOnPiX}
The previous Theorem shows that the order induced by $\pi:\mPV\to\PTV$ on $\PTV X=\URels{\TV}(X,1)$ is the pointwise order of $\V$-relations. With respect to this order, $\PTV X$ is then a complete lattice (whose infimum operation can easily be checked to be the infimum in $\Rels{\V}$). Moreover, the right $\V$-action on $\PTV X$ is given by
\[
\varphi\ast v=\yonmult_X\cdot\pi_{T X}(\eta_{\PTV X}(\varphi)\otimes v)=e_1^\circ\cdot\eT(\varphi\otimes v)\cdot m_X^\circ.
\]
for all $\varphi\in\PTV X$ and $v\in\V$. The internal hom can then be obtained by noticing that for $\psi\in\PTV X$,
\[
\varphi\ast v\le\psi\iff e_1^\circ\cdot\eT(\varphi\otimes v)\cdot m_X^\circ\le\psi\iff\varphi\otimes v\le\psi;
\]
hence, proceeding as in~\ref{ssec:Free}, we obtain
\begin{align*}
K((\yonmult_X\cdot\pi_{T X})^\dashv)(\varphi,\psi)&=\psi\Vhoml\varphi
\end{align*}
for all $\varphi,\psi\in\Rels{\V}(TX,1)$. In other words, the internal hom of the $\V$-category $\URels{(\mT,\V)}(X,1)$ is obtained by restriction of the internal hom of $\Rels{V}(TX,1)$; this also justifies \textit{a posteriori} the notation used for the internal hom of $TX$ defined in~\ref{ssec:Free}.
\end{rem}

\subsection{The Yoneda embedding}
In our context, the Yoneda lemma takes on the following form.

\begin{lem}\label{lem:Yoneda}
Let $\mPTV$ be the discrete presheaf monad of an associative lax extension $\emT$ of $\mT$ to $\Rels{\V}$. Then
\[
(\psi\Vhoml x^\sharp)=\psi(x,\star)
\]
for all $\psi\in\PTV X$ and $x\in X$.
\end{lem}

\begin{proof}
Remark~\ref{rem:ActionOnPiX} with~\ref{ssec:Free} yields
\[
(\psi\Vhoml\varphi)=\bigwedge\{\varphi(y,\star)\Vhomr \psi(y,\star)\mid y\in X\}
\]
for all $\psi\in\PTV X$. Setting $\varphi=x^\sharp$ (for $x\in X$), and observing that $k\le x^\circ\cdot\eT1_X(x,\star)=x^\sharp(x,\star)$ we obtain
\[
(\psi\Vhoml x^\sharp)\le(x^\sharp(x,\star)\Vhomr \psi(x,\star))\le(k\Vhomr \psi(x,\star))=\psi(x,\star).
\]
Moreover, $\psi$ is unitary, so $\eT1_X(y,x)\otimes\psi(x,\star)\le\psi(y,\star)$ for all $y\in X$, that is,
\[
\psi(x,\star)\le(\eT1_X(y,x)\Vhomr \psi(y,\star)),
\]
for all $y\in X$, and therefore $\psi(x,\star)\le(\psi\Vhoml x^\sharp)$.
\end{proof}

The Yoneda embedding then follows.

\begin{prop}\label{prop:Yoneda}
Let $\mPTV$ be the discrete presheaf monad of an associative lax extension $\emT$ of $\mT$ to $\Rels{\V}$. Then
\[
\eT r(\fx,\fy)=\kePTV r(\fx^\sharp,\fy^\sharp)
\]
for all $\V$-relations $r:X\relto Y$, and $\fx\in TX$, $\fy\in TY$.
\end{prop}

\begin{proof}
A direct computation shows that $r^\pi(\psi)=\psi\cdot\eT r$ for all $\psi\in\PTV X$. We can then use Lemma~\ref{lem:Yoneda} to write
\[
\kePTV r(\fx^\sharp,\fy^\sharp)=(r^\pi(\fy^\sharp\:\!)\Vhoml\fx^\sharp)=r^\pi(\fy^\sharp\:\!)(\fx,\star)=\fy^\sharp\cdot\eT r(\fx,\star)=\eT r(\fx,\fy)
\]
for all $\fx\in TX$, $\fy\in TY$.
\end{proof}

%%%%%%%%%%%%%%%%%%%%%%%%%%%%%%%%%%%%%%%%%

\section{Categories of lax extensions}

\subsection{The category of associative lax extensions}
Given a quantale $\V$, an object of the category\footnote{Strictly speaking, $\LaxExts{\V}$ is a metacategory, but we ignore such size issues here for questions of readability.}
\[
\LaxExts{\V}
\]
is a pair $(\mT,\emT)$ consisting of a monad $\mT$ on $\Set$ and of an associative lax extension $\emT$ of $\mT$ to $\Rels{\V}$. A morphism $\alpha:(\mS,\emS)\to(\mT,\emT)$ is a monad morphism $\alpha:\mS\to\mT$ such that $(m\cdot T\alpha)^\circ\cdot\eT1:\eT\relto\eT\eS$ is a natural transformation in $\Rels{\V}$.

\begin{rem}
Since $m^\circ:\eT\relto\eT\eT$ is a natural transformation, the identity $1_\mT:(\mT,\emT)\to(\mT,\emT)$ is a morphism of $\LaxExts{\V}$. Moreover, morphisms of $\LaxExts{\V}$ compose. Indeed, consider monad morphisms $\alpha:\mS\to\mT$ and $\beta:\mR\to\mS$ such that $(m\cdot T\alpha)^\circ\cdot\eT1:\eT\relto\eT\eS$ and $(n\cdot S\beta)^\circ\cdot\eS1:\eS\relto\eS\eR$ are natural transformations in $\Rels{\V}$. Then $\alpha\cdot\beta:\mR\to\mT$ is a monad morphism, and for a $\V$-relation $r:X\relto Y$, one obtains that $(m\cdot T(\alpha\cdot\beta))^\circ\cdot\eT1:\eT\relto\eT\eR$ is a natural transformation by using the hypotheses on $\alpha$ and $\beta$, together with the fact that $m_X^\circ\cdot\eT1_X=\eT\eT1_{X}\cdot m_X^\circ\cdot\eT1_{X}=\eT1_{TX}\cdot m_X^\circ\cdot\eT1_{X}$ and that $\emS$ and $\emT$ are associative lax extensions:
\begin{align*}
(m_Y\cdot&T(\alpha_Y\cdot\beta_Y))^\circ\cdot\eT1_Y\cdot\eT r\\
&=(T\beta_Y)^\circ\cdot\eT\eS r\cdot(m_X\cdot T\alpha_X)^\circ\cdot\eT1_X\\
&=(Td_{SRY})^\circ\cdot(TS\beta_Y)^\circ\cdot(Tn_Y)^\circ\cdot\eT\eS r\cdot(m_X\cdot T\alpha_X)^\circ\cdot\eT1_X\\
&=(Td_{SRY})^\circ\cdot\eT((S\beta_Y)^\circ\cdot n_Y^\circ\cdot\eS r)\cdot(m_X\cdot T\alpha_X)^\circ\cdot\eT1_X\\
&=(Td_{SRY})^\circ\cdot\eT(\eS\eR r\cdot(S\beta_{X})^\circ\cdot n_{X}^\circ\cdot\eS1_{X})\cdot(m_X\cdot T\alpha_X)^\circ\cdot\eT1_X\\
&=(Td_{SRY})^\circ\cdot\eT(\eS\eR r\cdot(S\beta_{X})^\circ\cdot n_{X}^\circ)\cdot(m_X\cdot T\alpha_X)^\circ\cdot\eT1_X\\
&=(Td_{SRY})^\circ\cdot\eT(\eS\eR r\cdot(S\beta_{X})^\circ\cdot n_{X}^\circ)\cdot\eT(\alpha_X^\circ)\cdot m_X^\circ\cdot\eT1_X\\
&=(Td_{SRY})^\circ\cdot\eT(\eS\eR r\cdot(S\beta_{X})^\circ\cdot n_{X}^\circ\cdot\alpha_X^\circ)\cdot m_X^\circ\cdot\eT1_X\\
&=(Td_{SRY})^\circ\cdot\eT(\eS\eR r\cdot\alpha_{RX}^\circ\cdot(T\beta_X)^\circ\cdot(T\alpha_X)^\circ\cdot m_X^\circ)\cdot m_X^\circ\cdot\eT1_X\\
&=(Td_{SRY})^\circ\cdot\eT(\eS\eR r\cdot\alpha_{RX}^\circ\cdot(T\beta_X)^\circ\cdot(T\alpha_X)^\circ\cdot m_X^\circ\cdot\eT1_X)\cdot m_X^\circ\cdot\eT1_X\\
&=(Td_{SRY})^\circ\cdot\eT\eS\eR r\cdot\eT(\alpha_{RX}^\circ)\cdot\eT\eT((\alpha\cdot\beta_X)^\circ)\cdot (Tm_X)^\circ\cdot m_X^\circ\cdot\eT1_X\\
&=(Td_{SRY})^\circ\cdot\eT\eS\eR r\cdot\eT(\alpha_{RX}^\circ)\cdot m_{RX}^\circ\cdot\eT((\alpha\cdot\beta_X)^\circ)\cdot m_X^\circ\cdot\eT1_X\\
&=\eT\eR r\cdot(m_X\cdot T(\alpha_X\cdot\beta_X))^\circ\cdot\eT1_X.
\end{align*}
\end{rem}

\subsection{The Kleisli extension as universal object}
We now proceed to describe an adjunction
\[
F\dashv G:(\mP_\V\comma\Mnds{\Set})\to\LaxExts{\V}
\]
that presents the comma category of monads under $\mP_\V$ as a full reflective subcategory of the category of associative lax extensions to $\Rels{\V}$ (Theorem~\ref{thm:TheAdjunction}). In particular, the discrete presheaf monad of an associative lax extension to $\Rels{\V}$ appears as its free $\V$-cocompletion (Proposition~\ref{prop:YonedaLaxExt}).

\begin{lem}\label{lem:URelDiag}
A morphism $\alpha:(\mS,\emS)\to(\mT,\emT)$ in $\LaxExts{\V}$ induces a functor $A:\URels{\SV}\to\URels{\TV}$ that makes the following diagram commute:
\[
\xymatrix@C=1em{\URels{\SV}^\op\ar[rr]^{A^\op}&&\URels{\TV}^\op\\
&\Set\ar[ul]^{(-)^\sharp}\ar[ur]_{(-)^\sharp}&}
\]
\end{lem}

\begin{proof}
For a $\SV$-relation $r:SX\relto Y$, we define the following $(\mT,\V)$-relation:
\[
Ar:=e_Y^\circ\cdot\eT r\cdot(m_X\cdot T\alpha_X)^\circ\cdot\eT 1_X:TX\relto Y.
\]
To verify that $A$ preserves Kleisli convolution, consider $\SV$-relations $r:X\krelto Y$ and $s:Y\krelto Z$. We first observe that by naturality of $(m\cdot T\alpha)^\circ\cdot\eT1$,
\begin{align*}
\eT1_{SX}\cdot (T\alpha_X)^\circ\cdot m_X^\circ\cdot\eT1_X&=\eT1_{SX}\cdot (T\alpha_X)^\circ\cdot m_X^\circ\cdot\eT1_X\cdot\eT1_X\\
&=\eT\eS1_{X}\cdot (T\alpha_X)^\circ\cdot m_X^\circ\cdot\eT1_X\\
&=(T\alpha_X)^\circ\cdot m_X^\circ\cdot\eT1_X
\end{align*}
and recall that that $m_X^\circ\cdot\eT1_X=\eT1_{TX}\cdot m_X^\circ\cdot\eT1_{X}$. These remarks allow us to establish the following chain of equalities:
\begin{align*}
As\kleisli Ar&=e_Z^\circ\cdot\eT s\cdot(m_Y\cdot T\alpha_Y)^\circ\cdot\eT(e_Y^\circ\cdot\eT r\cdot(m_X\cdot T\alpha_X)^\circ\cdot\eT 1_X)\cdot m_X^\circ\\
&=e_Z^\circ\cdot\eT s\cdot(m_Y\cdot T\alpha_Y)^\circ\cdot(Te_Y)^\circ\cdot\eT\eT r\cdot(TT\alpha_X)^\circ\cdot(Tm_X)^\circ\cdot\eT\eT 1_X\cdot m_X^\circ\\
&=e_Z^\circ\cdot\eT s\cdot(m_Y\cdot T\alpha_Y)^\circ\cdot(Te_Y)^\circ\cdot\eT\eT r\cdot(TT\alpha_X)^\circ\cdot(m_{TX})^\circ\cdot m_X^\circ\cdot\eT 1_X\\
&=e_Z^\circ\cdot\eT s\cdot(m_Y\cdot T\alpha_Y)^\circ\cdot(Te_Y)^\circ\cdot m_Y^\circ\cdot\eT r\cdot(T\alpha_X)^\circ\cdot m_X^\circ\cdot\eT 1_X\\
&=e_Z^\circ\cdot\eT s\cdot(m_Y\cdot T\alpha_Y)^\circ\cdot\eT r\cdot(T\alpha_X)^\circ\cdot m_X^\circ\cdot\eT 1_X\\
&=e_Z^\circ\cdot\eT(s\cdot\eS r)\cdot(m_{SX}\cdot T\alpha_{SX})^\circ\cdot\eT 1_{SX}\cdot(T\alpha_X)^\circ\cdot m_X^\circ\cdot\eT 1_X\\
&=e_Z^\circ\cdot\eT(s\cdot\eS r)\cdot(T\alpha_{SX})^\circ\cdot m_{SX}^\circ\cdot(T\alpha_X)^\circ\cdot m_X^\circ\cdot\eT 1_X\\
&=e_Z^\circ\cdot\eT(s\cdot\eS r)\cdot(T\alpha_{SX})^\circ\cdot(TT\alpha_X)^\circ\cdot(Tm_X)^\circ\cdot m_X^\circ\cdot\eT 1_X\\
&=e_Z^\circ\cdot\eT(s\cdot\eS r)\cdot(Tn_X)^\circ\cdot(T\alpha_X)^\circ\cdot\eT1_{TX}\cdot m_X^\circ\cdot\eT 1_X\\
&=e_Z^\circ\cdot\eT(s\kleisli r)\cdot\eT(\alpha_X^\circ)\cdot m_X^\circ\cdot\eT 1_X\\
&=e_Z^\circ\cdot\eT(s\kleisli r)\cdot(T\alpha_X)^\circ\cdot m_X^\circ\cdot\eT 1_X=A(s\kleisli r).
\end{align*}
Using that $\eT(e_Y^\circ)\cdot m_Y^\circ=\eT1_Y$, one then computes for a map $f:Y\to X$,
\[
A(d_Y^\circ\cdot\eS(f^\circ))=e_Y^\circ\cdot\eT(d_Y^\circ)\cdot\eT(\alpha_Y^\circ)\cdot m_Y^\circ\cdot\eT(f^\circ)=e_Y\cdot\eT(f^\circ),
\]
so that the stated diagram commutes and in particular $A(1_X^\sharp)=1_X^\sharp$. Finally, since $A$ preserves Kleisli convolution, one has for a unitary $\SV$-relation $r$ that
\[
Ar\kleisli 1_X^\sharp=Ar\kleisli A(1_X^\sharp)=A(r\kleisli 1_X^\sharp)=Ar
\]
and similarly $1_X^\sharp\kleisli Ar=Ar$, so $A$ is well-defined.
\end{proof}

\begin{prop}\label{prop:F}
Morphisms $\alpha:(\mS,\emS)\to(\mT,\emT)$ in $\LaxExts{\V}$ functorially determine  morphisms $\PTV(\alpha):\PTV(\mS,\emS)\to\PTV(\mT,\emT)$ in $(\mP_\V\comma\Mnds{\Set})$ via the components
\[
\PTV(\alpha)_X(r)=e_1^\circ\cdot\eT r\cdot(m_X\cdot T\alpha_X)^\circ\cdot\eT1_X\in\URels{\TV}(X,1)
\]
for all $r\in\URels{\SV}(X,1)$. That is, there is a functor
\[
F=\PTV:\LaxExts{\V}\to(\mP_\V\comma\Mnds{\Set}).
\]
\end{prop}

\begin{proof}
With the isomorphism $\URels{\TV}^\op\cong\Set_\PTV$ (Proposition~\ref{prop:NatTransfIntoPi}), the statement follows from Lemma~\ref{lem:URelDiag} together with the classical one-to-one functorial correspondence between monad morphisms and morphisms of Kleisli categories that commute with the respective left adjoint functors.
\end{proof}

\begin{prop}\label{prop:YonedaLaxExt}
Let $\emT$ be an associative lax extension of $\mT$ to $\Rels{\V}$, and $\PTV$ its associated discrete presheaf monad. Then the natural transformation $\Yoneda:T\to\PTV$, defined componentwise by
\[
 \Yoneda\!_X:TX\to\PTV X\ ,\quad\fx\mapsto\fx^\sharp=\fx^\circ\cdot \eT 1_X,
\]
yields a morphism $\Yoneda=\Yoneda\!_{(\mT,\emT)}:(\mT,\emT)\to(\mPTV,\kemPTV)$ in $\LaxExts{\V}$.
\end{prop}

\begin{proof}
The left adjoint functor $(-)^\sharp:\Set\to\URels{(\mT,\V)}^\op$ extends to a functor
\[
 (-)^\sharp:\Set_\mT\to\URels{(\mT,\V)}^\op
\]
sending $r:X\to TY$ to $r^\sharp:=r^\circ\cdot\eT 1_Y:Y\krelto X$, so the diagram
\[
\xymatrix{\Set_\mT\ar[rr]^-{(-)^\sharp} && \URels{(\mT,\V)}^\op\\
&\Set\ar[lu]^{F_\mT}\ar[ur]_{(-)^\sharp}}
\]
commutes. Hence, $(-)^\sharp:\Set_\mT\to\URels{(\mT,\V)}^\op$ induces a monad morphism $\Yoneda:\mT\to\PTV$ whose component $\Yoneda\!_X$ is the composite
\[
\xymatrix@C=4em{TX\ar[r]^-{\yoneda_{TX}}&\URels{(\mT,\V)}(TX,1)\ar[r]^-{(-)\kleisli {1_{TX}^\sharp}}&\URels{(\mT,\V)}(X,1),}
\]
that is, $\Yoneda\!_X(\fx)=\fx^\circ\cdot e_{TX}^\circ\cdot\eT1_{TX}\cdot\eT\eT 1_{X}\cdot m_X^\circ=\fx^\circ\cdot\eT 1_X$.

We are left to verify that $(\yonmult\cdot\PTV\Yoneda)^\circ\cdot\kePTV1:\kePTV\relto\kePTV\eT$ is natural in $\Rels{\V}$. For this, we denote by $\pi:\mPV\to\PTV$ the $\mPV$-structure of $\PTV$, and remark that for $\fx\in \PTV X$, $\fY\in \PTV TY$ and a $\V$-relation $r:X\relto Y$,
\begin{align*}
(\yonmult_Y\cdot \PTV\Yoneda_Y)^\circ\cdot\kePTV1_Y\cdot\kePTV r(\fx,\fY)&=\kePTV r(\fx,\yonmult_Y\cdot \PTV\Yoneda_Y(\fY))\\
&=((r^\pi\cdot \yonmult_Y\cdot \PTV\Yoneda_Y(\fY))\Vhoml\fx)
\end{align*}
and
\begin{align*}
\kePTV\eT r&\cdot (\yonmult_X\cdot \PTV\Yoneda_X)^\circ\cdot\kePTV1_X(\fx,\fY)\\
&=\Vee_{\fZ\in \PTV TX}(\yonmult_X\cdot \PTV\Yoneda_X(\fZ)\Vhoml\fx)\otimes((\eT r)^\pi(\fY)\Vhoml\fZ).
\end{align*}
By setting $\fZ=(\eT r)^\pi(\fY)$ in $\yonmult_X\cdot \PTV\Yoneda_X(\fZ)$, we have $k\le((\eT r)^\pi(\fY)\Vhoml\fZ)$ and compute for a unitary $(\mT,\V)$-relation $\Psi\in TY\krelto 1$,
\begin{align*}
\yonmult_X\cdot \PTV\Yoneda_X\cdot(\eT r)^\pi(\Psi)&=\yonmult_X\cdot \PTV\Yoneda_X\cdot \yonmult_{TX}\cdot \PTV(\pi_{TX}\cdot(\eT r)^\flat)(\Psi)\\
&=\yonmult_X\cdot \PTV\Yoneda_X\cdot \yonmult_{TX}\cdot \PTV(e_{TX}^\circ\cdot\eT\eT r)^\flat(\Psi)\\
&=\yonmult_X\cdot \PTV\Yoneda_X(\Psi\kleisli(((e_{TX}^\circ\cdot\eT\eT r)^\flat)^\sharp\kleisli\epsilon_{TY}))\\
&=\yonmult_X\cdot \PTV\Yoneda_X(\Psi\cdot\eT\eT r)\\
&=\Psi\cdot m_X^\circ\cdot\eT r=r^\pi\cdot \yonmult_Y\cdot \PTV\Yoneda_Y(\Psi)
\end{align*}
by using twice that $\Yoneda_{Y}^\sharp\kleisli\epsilon_{Y}=\Yoneda_{Y}^\circ\cdot\epsilon_{Y}=\eT1_{Y}$; thus, $(\yonmult_Y\cdot \PTV\Yoneda_Y)^\circ\cdot\kePTV1_Y\cdot\kePTV r\le\kePTV\eT r\cdot (\yonmult_X\cdot \PTV\Yoneda_X)^\circ\cdot\kePTV1_X$. To show the other inequality, it suffices to show that $\Yoneda^\circ:\kePTV\to\eT$ is a lax natural transformation in $\Rels{\V}$; indeed, in this case, we would have
\begin{align*}
\kePTV\eT r\cdot (\yonmult_X\cdot \PTV\Yoneda_X)^\circ\cdot\kePTV1_X&=\kePTV\eT r\cdot \kePTV(\Yoneda_X^\circ)\cdot \yonmult_X^\circ\cdot\kePTV1_X\\
&\le\kePTV(\Yoneda_Y^\circ)\cdot\kePTV\kePTV r\cdot \yonmult_X^\circ\cdot\kePTV1_X\\
&=(\yonmult_Y\cdot \PTV\Yoneda_Y)^\circ\cdot\kePTV1_Y\cdot\kePTV r.
\end{align*}
But lax naturality of $\Yoneda^\circ$, or equivalently, the inequality $\eT r(\fx,\fy)\le(r^\pi(\Yoneda_Y(\fy))\Vhoml\Yoneda_X(\fx))$ (for all $\V$-relations $r:X\relto Y$, and $\fx\in TX$, $\fy\in TY$), follows from Proposition~\ref{prop:Yoneda}.
%By definition of the internal hom, this condition is equivalent to 
%\[
%\Yoneda_X(\fx)\ast\eT r(\fx,\fy)\le\yonmult_X\cdot\PTV(\pi_X\cdot r^\flat)(\Yoneda_Y(\fy)).
%\]
%Since
%\begin{align*}
%\yonmult_X\cdot\PTV(\pi_X\cdot r^\flat)(\Yoneda_Y(\fy))&=\Yoneda_Y(\fy)\kleisli((\pi_X\cdot r^\flat)^\sharp\kleisli\epsilon_X)\\
%&=\fy^\circ\cdot\eT((\pi_X\cdot r^\flat)^\circ\cdot\epsilon_X)\cdot m_X^\circ\\
%&=\fy^\circ\cdot\eT(e_X\cdot\eT r)\cdot m_X^\circ=\fy^\circ\cdot\eT r,
%\end{align*}
%the sought inequality then follows by Remark~\ref{rem:ActionOnPiX}, since
%\[
%\Yoneda_X(\fx)\ast\eT r(\fx,\fy)=e_1^\circ\cdot\eT(\fx^\circ\cdot\eT1_X\otimes\eT r(\fx,\fy))\cdot m_X^\circ\le e_1^\circ\cdot\eT(\fy^\circ\cdot\eT r)\cdot m_X^\circ=\fy^\circ\cdot\eT r.
%\]
\end{proof}

\begin{lem}\label{lem:MndMorphLaxExtMorph}
A morphism of $\V$-power-enriched monads $\alpha:(\mS,\sigma)\to(\mT,\tau)$ satisfies
\[
\keS r\cdot\alpha_X^\circ\le\alpha_Y^\circ\cdot\keT r
\]
for all $\V$-relations $r:X\relto Y$.
\end{lem}

\begin{proof}
We first remark that
\[
r^\tau\cdot\alpha_Y=\alpha_X\cdot r^\sigma
\]
The pointwise version of the stated condition reads as $\keS r(\fx,\fy)\le\keT r(\alpha_X(\fx),\alpha_Y(\fy))$, that is,
\[
(r^\sigma(\fy)\Vhoml\fx)\le(\alpha_X\cdot r^\sigma(\fy)\Vhoml\alpha_X(\fx)).
\]
This condition follows from the fact that $\alpha_X:SX\to TX$ is a morphisms of the $\mPV$-algebras $(SX,n_X\cdot\sigma_{SX})$ and $(TX,m_X\cdot\tau_{TX})$, and this in turn follows from $\alpha:\mS\to\mT$ being a morphism in $(\mPV\comma\Mnds{\Set})$.
\end{proof}

\begin{prop}\label{prop:G}
Morphisms $\alpha:(\mS,\sigma)\to(\mT,\tau)$ in $(\mP_\V\comma\Mnds{\Set})$ functorially determine morphisms $\alpha:(\mS,\kemS)\to(\mT,\kemT)$ in $\LaxExts{\V}$ (where the $\V$-power-enriched monads $\mS$ and $\mT$ are equipped with their respective Kleisli extensions). More precisely, there is a functor
\[
G:(\mP_\V\comma\Mnds{\Set})\to\LaxExts{\V}
\]
that commutes with the respective forgetful functors to $\Mnds{\Set}$. 
\end{prop}

\begin{proof}
We only need to verify that a monad morphism $\alpha:\mS\to\mT$ such that $\tau=\alpha\cdot\sigma$ yields a natural transformation $(m\cdot T\alpha)^\circ\cdot\keT1:\keT\to\keT\keS$ in $\Rels{\V}$. Here, we use the notations $\mT=(T,m,e)$, $\mS=(S,n,d)$, and $\mPV=(\PV,\mu,\eta)$.

First, we remark that for $\fx\in TX$, $\fY\in TSY$ and a $\V$-relation $r:X\relto Y$,
\[
(m_Y\cdot T\alpha_Y)^\circ\cdot\keT1_Y\cdot\keT r(\fx,\fY)=\keT r(\fx,m_Y\cdot T\alpha_Y(\fY))=((r^\tau\cdot m_Y\cdot T\alpha_Y(\fY))\Vhoml\fx)
\]
and
\[
\keT\keS r\cdot (m_X\cdot T\alpha_X)^\circ\cdot\keT1_X(\fx,\fY)=\Vee_{\fZ\in TSX}(m_X\cdot T\alpha_X(\fZ)\Vhoml\fx)\otimes((\keS r)^\tau(\fY)\Vhoml\fZ).
\]
Since
\begin{align*}
m_X\cdot T\alpha_X\cdot(\keS r)^\tau&=m_X\cdot T\alpha_X\cdot m_{SX}\cdot T(\alpha_{SX}\cdot\sigma_{SX}\cdot(\keS r)^\flat)\\
&=m_X\cdot T(m_X\cdot T\alpha_X\cdot\alpha_{SX}\cdot\sigma_{SX}\cdot(n_X\cdot\sigma_{SX})^\dashv\cdot r^\sigma)\\
&=m_X\cdot T(\alpha_{X}\cdot n_X\cdot\sigma_{SX}\cdot(n_X\cdot\sigma_{SX})^\dashv\cdot r^\sigma)\\
&=m_X\cdot T(\alpha_{X}\cdot r^\sigma)\\
&=m_X\cdot T(r^\tau\cdot\alpha_Y)=r^\tau\cdot m_Y\cdot T\alpha_Y,
\end{align*}
we have $(m_Y\cdot T\alpha_Y)^\circ\cdot\keT1_Y\cdot\keT r\le\keT\keS r\cdot (m_X\cdot T\alpha_X)^\circ\cdot\keT1_X$. Moreover, by Lemma~\ref{lem:MndMorphLaxExtMorph}, we also have the other inequality, and the sought naturality follows:
\begin{align*}
\keT\keS r\cdot (m_X\cdot T\alpha_X)^\circ\cdot\keT1_X&=\keT\keS r\cdot \keT(\alpha_X^\circ)\cdot m_X^\circ\cdot\keT1_X\\
&\le\keT(\alpha_Y^\circ)\cdot\keT\keT r\cdot m_X^\circ\cdot\keT1_X=(m_Y\cdot T\alpha_Y)^\circ\cdot\keT1_Y\cdot\keT r.
\end{align*}
\end{proof}

\begin{thm}\label{thm:TheAdjunction}
The functor $G:(\mP_\V\comma\Mnds{\Set})\to\LaxExts{\V}$ (see Proposition~\ref{prop:G}) is a full and faithful embedding, and has $F:\LaxExts{\V}\to(\mP_\V\comma\Mnds{\Set})$ as left adjoint. 
\end{thm}

\begin{proof}
To prove the statement, we show that $F$ is left adjoint to $G$ and that the counit of this adjunction is an isomorphism. Hence, consider a $\mPV$-structures monad $(\mT,\tau)$ (with its Kleisli extension $\emT$ to $\Rels{\V}$). By Proposition~\ref{prop:NatTransfIntoPi}, Theorem~\ref{thm:ConvAndNbhd} and the one-to-one correspondence between monad morphisms and functors between Kleisli categories, all the triangles but the lower-right one in the diagram
\[
\xymatrix{&\Set\ar[d]^{(-)^\circ}\ar@/_1.2pc/[ddl]_{F_\PTV}\ar@/^1.2pc/[ddr]^{F_\mT}&\\
&\rule{5ex}{0pt}\Rels{\V}^\op\rule{5ex}{0pt}\ar[d]^{(-)_\sharp}\ar[dl]|{Q^\inv(-)_\sharp\ }\ar[dr]|{\rule[-1ex]{1ex}{0pt}\Set_\tau(-)^\flat}&\\
\Set_\PTV\ar[r]^-{Q}&\URels{(\mT,\V)}^\op\ar[r]^-{\nbhd}&\Set_\mT}
\]
commute. For the last triangle, a computation similar to the last displayed equation in the proof of Theorem~\ref{thm:ConvAndNbhd} shows that for any $\V$-relation $r:X\relto Y$,
\[
\nbhd(r_\sharp)=m_Y\cdot\tau_{TY}\cdot(e_Y^\circ\cdot\keT r)^\flat=\tau_X\cdot r^\flat=\Set_\tau(r^\flat).
\]
One deduces that the following diagram commutes
\[
\xymatrix@C=2em{&\mPV\ar[dl]_{\pi}\ar[dr]^{\tau}&\\
\PTV\ar[rr]^{\kappa}&&\mT,}
\]
where $\kappa$ is the monad morphism induced by the composition of $\nbhd$ and $Q$. Since these two functors are isomorphisms, $\kappa=\kappa_{(\mT,\tau)}:FG(\mT,\tau)\to(\mT,\tau)$ is itself an isomorphism in $(\mP_\V\comma\Mnds{\Set})$. In fact, one computes $\kappa=m\cdot\tau T$.

We are left to verify that our candidates for the unit and counit of the adjunction, $\Yoneda:1\to GF$ and $\kappa:FG\to 1$ respectively, satisfy the triangular identities. For $\fy\in TX$, we use that $(TX, m_X\cdot\tau_{TX})$ is a $\mPV$-algebra and that $\Vee_{\fy\in TX}\fy\ast(\fx\Vhoml\fy)=\fx$ (see~\ref{ssec:Free}) to write
\[
\kappa_X\cdot\Yoneda_X(\fx)=m_X\cdot\tau_{TX}(\fx^\circ\cdot\keT1_X)=\Vee_{\fy\in TX}\fy\ast(\fx^\circ\cdot\keT1_X(\fy,\star))=\Vee_{\fy\in TX}\fy\ast(\fx\Vhoml\fy)=\fx
\]
for all $\fx\in TX$, that is, $G\kappa\cdot\Yoneda G=1$. For the other identity, we note that for all $r\in\URels{(\mT,\V)}(X,1)$,
\begin{align*}
\kappa_{X}&\cdot\PTV(\Yoneda)_X(r)\\
&=m_X\cdot\tau_{TX}(\yoneda_1^\circ\cdot\kePTV r\cdot(\yonmult_X\cdot\PTV\Yoneda_X)^\circ\cdot\kePTV1_X)\\
&=\Vee_{\phi\in\PTV X}\phi\ast(\yoneda_1^\circ\cdot\kePTV r\cdot(\yonmult_X\cdot\PTV\Yoneda_X)^\circ\cdot\kePTV1_X(\phi,\star))\\
&=\Vee_{\phi\in\PTV X,\Phi\in\PTV T X}\phi\ast(\kePTV1_X(\phi,\yonmult_X\cdot\PTV\Yoneda_X(\Phi))\otimes\kePTV r(\Phi,\yoneda_1(\star)))\\
&=\Vee_{\phi\in\PTV X,\Phi\in\PTV T X}(\phi\ast(\yonmult_X\cdot\PTV\Yoneda_X(\Phi)\Vhoml\phi))\ast(r^\pi(\yoneda_1(\star))\Vhoml\Phi)\\
&=\Vee_{\Phi\in\PTV T X}(\yonmult_X\cdot\PTV\Yoneda_X(\Phi))\ast(\pi_{TX}(r)\Vhoml\Phi)\\
&=\Vee_{\Phi\in\PTV T X}(\Phi\cdot m_X^\circ\cdot\eT1_X)\ast((e_1^\circ\cdot\eT r)\Vhoml\Phi).
\end{align*}
Since $\Phi=\pi_{TX}(r)=e_1^\circ\cdot\eT r$ is a possible value for $\Phi$ running through $\PTV TX$, we obtain
\[
r=e_1^\circ\cdot\eT r\cdot m_X^\circ\cdot\eT1_X\le\kappa_{X}\cdot\PTV(\Yoneda)_X(r).
\] 
Moreover, $\yonmult_X\cdot\PTV\Yoneda_X=(-)\cdot m_X^\circ\cdot\eT 1_X:\PTV TX\to\PTV X$ is a $\mPV$-algebra morphism, and therefore also a $\V$-functor of the underling $\V$-categories, so
\[
((e_1^\circ\cdot\eT r)\Vhoml\Phi)\le((e_1^\circ\cdot\eT r\cdot m_X^\circ\cdot\eT 1_X)\Vhoml(\Phi\cdot m_X^\circ\cdot\eT 1_X))=(r\Vhoml(\Phi\cdot m_X^\circ\cdot\eT 1_X));
\]
hence, $(\Phi\cdot m_X^\circ\cdot\eT1_X)\ast((e_1^\circ\cdot\eT r)\Vhoml\Phi)\le r$ for all $\Phi\in\PTV TX$, and
\[
\kappa_{X}\cdot\PTV(\Yoneda)_X(r)\le r
\]
holds. This shows $\kappa F\cdot F\Yoneda=1$.
\end{proof}

\section{Lax algebras}

\subsection{Categories of monoids}
Let $\emT$ be an associative lax extension of $\mT=(T,m,e)$ to $\Rels{\V}$. A \df{$(\mT,\V)$-category} (generically, a \df{lax algebra}) on a set $X$ is a monoid in $\URels{(\mT,\V)}(X,X)$. In other words, a lax algebra $(X,a)$ is a set $X$ with a $(\mT,\V)$-relation $a:X\krelto X$ satisfying
\[
1_X^\sharp\le a\qquad\text{and}\qquad a\kleisli a\le a,
\]
or equivalently, $e_X^\circ\le a$ and $a\cdot\eT a\cdot m_X^\circ\le a$. These conditions imply in particular that $a$ is unitary and idempotent:
\[
1_X^\sharp\kleisli a=a=a\kleisli 1_X^\sharp\qquad\text{and}\qquad a\kleisli a=a
\]
(see \cite[Section~III.1.8]{Cle/al:14} for details). A \df{$(\mT,\V)$-functor} between $(\mT,\V)$ categories $(X,a)$ and $(Y,b)$ is a map $f:X\to Y$ such that
\[
a\kleisli f^\sharp\le f^\sharp\kleisli b,
\]
or equivalently, $a\le f^\circ\cdot b\cdot Tf$. The category of $(\mT,\V)$-categories with $(\mT,\V)$-functors as morphisms is denoted by
\[
\Cats{(\mT,\V)}.
\]
By Theorem~\ref{thm:KleisliOfPi}, a monoid in $\URels{(\mT,\V)}(X,X)$ is a monoid in $\Set_\mPTV(X,X)$. More generally, given a $\two$-enriched monad $\mT=(T,m,e)$, a \df{$\mT$-monoid} (or a \df{Kleisli monoid}) is a set $X$ with a map $\nu:X\to TX$ such that
\[
e_X\le\nu\qquad\text{and}\qquad\nu\kleisli\nu\le\nu.
\]
A morphism between $\mT$-monoids $(X,\nu)$ and $(Y,\xi)$ is a map $f:X\to Y$ satisfying
\[
f_\natural\kleisli\nu\le\xi\kleisli f_\natural,
\]
where $f_\natural=e_Y\cdot f$; equivalently, this expression can be written $Tf\cdot\nu\le\xi\cdot f$. The category of $\mT$-monoids and their morphisms is denoted by
\[
\Mons{\mT}.
\]
In the general case where $\mT$ is a monad on $\Set$ (not necessarily $\two$-enriched) equipped with an an associative lax extension $\emT$ to $\Rels{\V}$, the discrete presheaf monad $\mPTV$ is always $\V$-power-enriched---and in particular $\two$-enriched. Via the order-isomorphism of Theorem~\ref{thm:KleisliOfPi}, it is easily verified that the morphisms of monoids correspond, so there is an isomorphism
\[
\Cats{(\mT,\V)}\cong\Mons{\mPTV}.
\]

\subsection{The Kleisli extension and monoids}
Recall from Theorem~\ref{thm:ConvAndNbhd} that if $\mT$ is a $\V$-power-enriched monad, there is an order-isomorphism
\[
\Set_\mT\cong\URels{(\mT,\V)}^\op
\]
when $\mT$ comes with its Kleisli extension $\kemT$ to $\Rels{\V}$. By the $\V$-enrichment, there is a monad morphism $\tau:\mPV\to\mT$, and the order on the hom-sets of $\Set_\mT$ is induced by the order on the sets $TX$, that is, by the semilattice structure. In other words, the $\two$-enrichment $\tau\cdot\iota:\mP\to\mT$ yields the same ordered category $\Set_\mT$ as the $\V$-enrichment $\tau:\mPV\to\mT$ (here, $\iota$ is the monad morphism $\iota:\mP\to\mPV$ given by the canonical quantale homomorphism $\two\to\V$ described in Example~\ref{ex:quantales}\ref{ex:quantales:two}).

\begin{prop}
Let $\mT$ a monad on $\Set$, and $\kappa:\mPW\to\mPV$, $\tau:\mPV\to\mT$ monad morphisms such that $(\mT,\tau)$ is $\V$-power-enriched. Then $(\mT,\tau\cdot\kappa)$ is $\W$-enriched, and there is an isomorphism
\[
\Cats{(\mT,\V)}\cong\Cats{(\mT,\W)},
\]
where $(\mT,\tau)$ and $(\mT,\tau\cdot\kappa)$ are equipped with their respective Kleisli extensions $\kemT_\V$ to $\Rels{\V}$, and $\kemT_\W$ to $\Rels{\W}$.
\end{prop}

\begin{proof}
By Theorem~\ref{thm:ConvAndNbhd}, there are order-isomorphisms
\[
\URels{(\mT,\W)}^\op\cong\Set_\mT\cong\URels{(\mT,\V)}^\op.
\]
Hence, monoids in the corresponding hom-sets are in bijective correspondence, and so are morphisms between these, that is, 
\[
\Cats{(\mT,\V)}\cong\Mons{\mT}\cong\Cats{(\mT,\W)},
\]
which proves the claim.
\end{proof}

%%%%%%%%%%%%%%%%%%%%%%%%%%%%%%%%%%%%%%%%%%%%%%%%%%%%%%%%%%
%%%%%%%%%%%%%%%%%%%%%%%%%%%%%%%%%%%%%%%%%%%%%%%%%%%%%%%%%%

\bibliographystyle{plain}

\begin{thebibliography}{10}

\bibitem{Bar:70}
M.~Barr.
\newblock Relational algebras.
\newblock In {\em Reports of the {M}idwest {C}ategory {S}eminar, {IV}}, volume
  137 of {\em Lect. Notes Math.}, pages 39--55, Berlin, 1970. Springer.

\bibitem{Cle/al:14}
M.M. Clementino, E.~Colebunders, D.~Hofmann, R.~Lowen, R.~Lucyshyn-Wright, G.J.
  Seal, and W.~Tholen.
\newblock {\em Monoidal Topology: A Categorical Approach to Order, Metric and
  Topology}.
\newblock Cambridge University Press, Cambridge, 2014.

\bibitem{CleHof:03}
M.M. Clementino and D.~Hofmann.
\newblock Topological features of lax algebras.
\newblock {\em Appl. Categ. Structures}, 11(3):267--286, 2003.

\bibitem{CleHofTho:04}
M.M. Clementino, D.~Hofmann, and W.~Tholen.
\newblock One setting for all: Metric, topology, uniformity, approach
  structure.
\newblock {\em Appl. Categ. Structures}, 12(2):127--154, 2004.

\bibitem{CleTho:03}
M.M. Clementino and W.~Tholen.
\newblock Metric, topology and multicategory---a common approach.
\newblock {\em J. Pure Appl. Algebra}, 179(1-2):13--47, 2003.

\bibitem{EnqSac:14}
S.~Enqvist and J.~Sack.
\newblock A coalgebraic view of characteristic formulas in equational modal
  fixed point logics.
\newblock In {\em Coalgebraic methods in computer science}, Lecture Notes in
  Comput. Sci., pages 98--117. Springer, 2014.

\bibitem{Hof:07}
D.~Hofmann.
\newblock Topological theories and closed objects.
\newblock {\em Adv. Math.}, 215(2):789--824, 2007.

\bibitem{LowVro:08}
R.~Lowen and T.~Vroegrijk.
\newblock A new lax algebraic characterization of approach spaces.
\newblock In {\em Theory and Applications of Proximity, Nearness and
  Uniformity}, volume~22 of {\em Quad. Mat.}, pages 199--232. Dept. Math.,
  Seconda Univ. Napoli, Caserta, 2008.

\bibitem{MarVen:12}
J.~Marti and Y.~Venema.
\newblock Lax extensions of coalgebra functors.
\newblock In {\em Coalgebraic methods in computer science}, volume 7399 of {\em
  Lecture Notes in Comput. Sci.}, pages 150--169. Springer, 2012.

\bibitem{PedTho:89}
M.C. Pedicchio and W.~Tholen.
\newblock Multiplicative structures over sup-lattices.
\newblock {\em Arch. Math.}, 25(1-2):107--114, 1989.

\bibitem{Sea:09}
G.J. Seal.
\newblock A {K}leisli-based approach to lax algebras.
\newblock {\em Appl. Categ. Structures}, 17(1):75--89, 2009.

\bibitem{Sea:10}
G.J. Seal.
\newblock Order-adjoint monads and injective objects.
\newblock {\em J. Pure Appl. Algebra}, 214(6):778--796, 2010.

\end{thebibliography}

\end{document}